\documentclass[a4paper]{amsart}
\usepackage{amsmath,amsthm,amssymb,latexsym,epic,bbm,comment,mathrsfs}
\usepackage{graphicx,enumerate,stmaryrd,xcolor}
\usepackage[all,2cell,knot]{xy}
\xyoption{2cell}

\newtheorem{theorem}{Theorem}
\newtheorem{lemma}[theorem]{Lemma}
\newtheorem{corollary}[theorem]{Corollary}
\newtheorem{proposition}[theorem]{Proposition}
\newtheorem{conjecture}[theorem]{Conjecture}

\newtheorem{remark}[theorem]{Remark}
\usepackage[all]{xy}
\usepackage[active]{srcltx}
\usepackage[parfill]{parskip}
\usepackage{enumerate}

\font\sc=rsfs10
\newcommand{\cC}{\sc\mbox{C}\hspace{1.0pt}}

\newcommand{\cJ}{\sc\mbox{J}\hspace{1.0pt}}

\newcommand{\cP}{\sc\mbox{P}\hspace{1.0pt}}

\font\scc=rsfs7
\newcommand{\ccC}{\scc\mbox{C}\hspace{1.0pt}}
\newcommand{\ccP}{\scc\mbox{P}\hspace{1.0pt}}

\newcommand{\ccJ}{\scc\mbox{J}\hspace{1.0pt}}

\begin{document}

\title[Kostant's problem for fully commutative permutations]
{Kostant's problem for fully commutative permutations}
\author[M.~Mackaay, V.~Mazorchuk and V.~Miemietz]{Marco Mackaay, 
Volodymyr Mazorchuk and Vanessa Miemietz}

\begin{abstract}
We give a complete combinatorial answer to Kostant's problem for simple 
highest weight modules indexed by fully commutative permutations.
We also propose a reformulation of Kostant's problem in the context of 
fiab bicategories and classify annihilators of simple objects 
in the principal birepresentations of such bicategories generalizing
the Barbasch--Vogan theorem for Lie algebras.
\end{abstract}

\maketitle

\section{Introduction and description of the results}\label{s1}

Kostant's problem, as defined and popularized by Joseph in \cite{Jo}, is
a famous open problem in representation theory of Lie algebras. It asks
for which simple modules $L$ over a semi-simple complex finite dimensional
Lie algebra the universal enveloping algebra surjects onto the space
of adjointly finite linear endomorphisms of $L$. Although several 
(both positive and negative) results are known, see Subsection~\ref{s3.6}
for a historical overview, the general case is wide open even for
simple highest weight modules in the principal block of BGG 
category $\mathcal{O}$.

Knowing the answer to Kostant's problem for a particular simple module 
is important for potential applications, for example, understanding 
the  structure of induced modules, see \cite{KhM,MS}, using the approach of
establishing equivalences of categories for Lie algebra modules proposed
in \cite{MiSo}.

As described in Subsection~\ref{s3.6}, the existing results on Kostant's 
problem can be divided into four classes:
\begin{itemize}
\item there are some classes of simple modules for which the answer is
known to be positive;
\item there are some classes of simple modules for which the answer is
known to be negative;
\item there are some results that relate the answer for one simple
module to the answer for another simple module;
\item there are some complete classification results for 
Lie algebras of small rank.
\end{itemize}
To the best of our knowledge, none of the existing results provides 
a complete solution to Kostant's problem for
some natural general class of simple modules which contains both
simple modules with positive and negative answers.
The main result of the present paper is the first result of this kind.

We look at a natural family of permutations (elements of the symmetric
group $S_n$) which are called {\em fully commutative}. The latter means that 
any two reduced expressions for such an element can be transformed one
into another just by using the commutativity relations for simple reflections.
The number of fully commutative elements in $S_n$ is given by the 
Catalan number $C_n=\frac{1}{n+1}\binom{2n}{n}$. Under the 
Robinson--Schensted correspondence, fully commutative elements are
exactly those permutations which correspond to two-row partitions.
From the point of view
of the Kazhdan--Lusztig combinatorics in type $A$, it is well known that 
fully commutative elements form a union of two-sided Kazhdan--Lusztig cells
and that, for general $n$, the value of Lusztig's $\mathbf{a}$-function
on fully commutative elements in $S_n$ can be arbitrarily large. This 
demonstrates that the class of fully commutative permutations is quite large. 
Fully commutative permutations naturally index a basis in the
Temperley-Lieb algebra, which is a certain quotient of the Hecke algebra of $S_n$.

The main result of the present paper, Theorem~\ref{thm5.2.1},
gives a complete combinatorial classification of the fully commutative
permutations such that Kostant's problem for the corresponding simple
modules in the principal block of BGG category $\mathcal{O}$ has 
positive answer (we call such elements {\em Kostant positive}). 
The answer is both quite beautiful and rather non-trivial.
We first define certain special fully commutative involutions.
These are permutations of the following form:

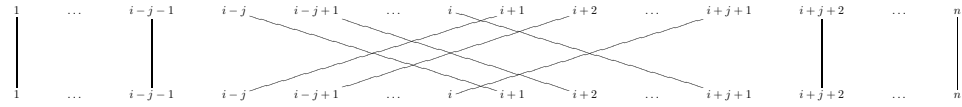
\resizebox{\textwidth}{!}{
$
\xymatrix@R=20mm@C=12mm{
1\ar@{-}[d]&\dots&i-j-1\ar@{-}[d]&i-j\ar@{-}[drrrr]&i-j+1\ar@{-}[drrrr]&\dots
&i\ar@{-}[drrrr]&i+1\ar@{-}[dllll]&i+2\ar@{-}[dllll]&\dots&i+j+1\ar@{-}[dllll]&
i+j+2\ar@{-}[d]&\dots&n\ar@{-}[d]\\
1&\dots&i-j-1&i-j&i-j+1&\dots&i&i+1&i+2&\dots&i+j+1&i+j+2&\dots&n
}
$
}

For a special involution, we call the area where this involution acts non-trivially 
its {\em  support}. In Theorem~\ref{thm5.2.1}, we show that a fully 
commutative involution $w$ is Kostant positive if and only if it is 
a product of distant special involutions, in the sense that the supports of any
two special factors in $w$ are separated in the diagram of $w$ by at least
one vertical line.

Our approach to Kostant's problem for fully commutative elements is
crucially based on the availability of the following three things:
\begin{itemize}
\item the explicit description of the Serre subcategory of category
$\mathcal{O}$ generated by the simple objects indexed by 
fully commutative elements and the
action of projective functors on this category, due to
Brundan and Stroppel, see \cite{BS};
\item K{\aa}hrstr{\"o}m's conjectural combinatorial reformulation of
Kostant's problem, see \cite[Conjecture~1.2]{KMM};
\item the $2$-representation theoretic approach to K{\aa}hrstr{\"o}m's 
conjecture in relation to Kostant's problem as developed 
by Ko, Mr{\dj}en and the second author in \cite[Theorem~B]{KMM}.
\end{itemize}
The idea behind the proof is to use K{\aa}hrstr{\"o}m's
combinatorial reformulation of Kostant's problem to reduce it
to the question of solvability of certain 
equations in the Temperley-Lieb algebra. Then we establish,
in the appropriate cases, the solvability of these 
equations by explicitly giving the solutions. In the remaining
cases, we establish the impossibility of solving these  
equations by a rather technical case-by-cases analysis.

Our answer to Kostant's problem for fully commutative elements
is explicit enough to compute its asymptotic behavior.
More precisely, in Theorem~\ref{thm5.6.1}, we prove the following:
\begin{itemize}
\item asymptotically, the answer to Kostant's problem for 
fully commutative elements is almost surely negative;
\item asymptotically, the answer to Kostant's problem for 
fully commutative involutions is almost surely negative;
\item asymptotically, the answer to Kostant's problem for 
fully commutative elements  (or involutions) of a fixed $\mathbf{a}$-value 
is almost surely positive.
\end{itemize}

From its very definition, Kostant's problem is intimately connected
with the problem of  understanding the annihilators of simple 
modules over semi-simple Lie algebras. These 
annihilators are classified combining two classical theorems
of Duflo and Barbasch--Vogan. In the last section of 
the paper, Section~\ref{s6}, we provide an analogue of the
Barbasch--Vogan Theorem for a class of 
bicategories known as fiab bicategories. We also
propose a reformulation of Kostant's problem in this more 
general context.

\vspace{1mm}

\subsection*{Acknowledgments}
The first author is partially supported by Funda{\c c}{\~a}o para a 
Ci{\^ e}ncia e a Tecnologia (Portugal),
projects PTDC/MAT-PUR/31089/2017 (Higher Structures
and Applications) and UID/MAT/04459/2013 (Center for Mathematical Analysis, Geometry 
and Dynamical Systems - CAMGSD).
The second author is partially supported by the Swedish Research Council.
The third author is partially supported by EPSRC grant
EP/S017216/1. We thank Darij Grinberg for helpful comments.
{ We thank the referees for helpful comments}

\section{Category $\mathcal{O}$ preliminaries}\label{s2}

\subsection{Setup}\label{s2.1}

We work over $\mathbb{C}$. 

Let $\mathfrak{g}$ denote a  semi-simple finite dimensional Lie algebra
with a fixed triangular decomposition 
\begin{equation}\label{eq1.1}
\mathfrak{g}=\mathfrak{n}_-\oplus \mathfrak{h}\oplus \mathfrak{n}_+. 
\end{equation}
Here $\mathfrak{h}$ is a Cartan subalgebra. We denote by $W$ the 
corresponding Weyl group, by $S$ the set of all  simple
reflections in $W$ associated to the triangular decomposition
\eqref{eq1.1} and by $w_0$ the longest element of $W$.

\subsection{Category $\mathcal{O}$}\label{s2.2}

Consider the BGG category $\mathcal{O}$ associated to the triangular decomposition
\eqref{eq1.1}, see \cite{BGG,Hu}. {Explicitly, this is the full subcategory of the category of all $\mathfrak{g}$-modules whose objects are finitely generated, diagonalizable with respect to the action of $\mathfrak{h}$ and locally finite with respect to the action of $\mathfrak{n}^+$.}

\subsection{Principal block}\label{s2.3}

Let $\mathcal{O}_0$ denote the principal  block of $\mathcal{O}$, that is the
direct summand containing the trivial $\mathfrak{g}$-module. Up to isomorphism,
the simple  objects in $\mathcal{O}_0$ are the simple highest weight modules
$L_w:=L(w\cdot 0)$, where $0\in \mathfrak{h}^*$ is the zero weight and
$w\in W$.  Here $\cdot$ is the usual dot-action of $W$ on $\mathfrak{h}^*$ defined, for 
$\lambda\in\mathfrak{h}^*$,  via $w\cdot \lambda=w(\lambda+\rho)-\rho$, where
$\rho$ is half of the sum of all  positive roots.

For $w\in W$, we denote by $\Delta_w$ the Verma cover 
and by $P_w$ the indecomposable projective cover of
$L_w$. Let $A$ be the opposite of the endomorphism algebra of the direct sum
of the $P_w$, taken over all $w\in W$. Then $A$ is a finite dimensional
associative algebra and the category $A$-mod of
finite dimensional $A$-modules is equivalent to $\mathcal{O}_0$.

\subsection{Projective functors}\label{s2.4}

For $w\in W$,  we denote by $\theta_w$ the unique, up to isomorphism,
indecomposable projective endofunctor of $\mathcal{O}_0$, in the sense
of \cite{BG}, {determined by the property
$\theta_w\, P_e\cong P_w$}.

We denote by $\cP$ the monoidal category of all projective functors 
acting on $\mathcal{O}_0$. Thanks to Soergel's combinatorial description,
$\cP$ is equivalent to the monoidal category of Soergel bimodules 
over the coinvariant algebra associated to $W$.

\subsection{Graded lift}\label{s2.5}

The algebra $A$ is Koszul, in particular, it has the corresponding
Koszul $\mathbb{Z}$-grading, see \cite{So,St}.
We denote by $\mathcal{O}_0^\mathbb{Z}$ the category of all
finite dimensional $\mathbb{Z}$-graded $A$-modules.

All structural modules in $\mathcal{O}_0$ admit graded lifts.
For indecomposable modules,  these graded lifts are unique, up to isomorphism
and grading shift. For indecomposable structural modules, there is a 
preferred graded lift, called standard lift.

Abusing notation, we will denote the ungraded 
objects in $\mathcal{O}_0$ by the same symbols as their standard graded lifts.
We use a similar convention for the graded version  of
projective functors, see \cite{St}, and denote by $\cP^\mathbb{Z}$
the monoidal category of all graded projective functors 
acting on $\mathcal{O}_0^\mathbb{Z}$.

\subsection{Hecke algebra combinatorics}\label{s2.6}

Let $\mathbf{H}=\mathbf{H}(W,S)$ be the Hecke algebra of $W$
over $\mathbb{Z}[v,v^{-1}]$ (in the normalization of \cite{So2}).
The algebra $\mathbf{H}$ has the standard basis $\{H_w\,:\,w\in W\}$  and the 
Kazhdan--Lusztig basis $\{\underline{H}_w\,:\,w\in W\}$, see \cite{KL}.

The Grothendieck group {$\mathbf{Gr}[\mathcal{O}_0^{\mathbb{Z}}]$} is isomorphic to
$\mathbf{H}$ by sending $[\Delta_w]$ to $H_w$. This isomorphism
sends $[P_w]$ to $\underline{H}_w$, see \cite{BB81,BK81}.

The split Grothendieck group $\mathbf{Gr}_\oplus[\cP^\mathbb{Z}]$
is isomorphic to $\mathbf{H}$ by sending $[\theta_w]$ to $\underline{H}_w$.
The action of $\cP^\mathbb{Z}$ on $\mathcal{O}_0^\mathbb{Z}$
decategorifies to the right regular $\mathbf{H}$-module, see \cite{So3}.

We denote by $\leq_L$, $\leq_R$ and $\leq_J$ the Kazhdan--Lusztig
left, right and two-sided pre-orders on $W$, respectively. 
The associated equivalence classes are called 
left, right and two-sided cells. Each left and right
cell contains a distinguished involution, also called the Duflo involution.
In type $A$, all involutions are Duflo involutions.

\subsection{Indecomposability conjecture}\label{s2.7}

The following conjecture is proposed in \cite[Conjecture~2]{KM}.

\begin{conjecture}\label{conj1}
If $\mathfrak{g}\cong\mathfrak{sl}_n$, then, for any 
$x,y\in W$, the module $\theta_x  L_y$ is either indecomposable or zero.
\end{conjecture}

We will denote by $\mathbf{KM}(x,y)\in\{\mathtt{true},\mathtt{false}\}$ the value of the claim
``the module $\theta_x  L_y$ is either indecomposable or zero''. We also denote
by $\mathbf{KM}(*,y)$ the conjunction of the $\mathbf{KM}(x,y)$ over all $x\in W$.

\section{Kostant's problem}\label{s3}

\subsection{Annihilators of simple modules}\label{s3.1}

Let $L$ be a simple $\mathfrak{g}$-module. By Dixmier--Schur's Lemma, $L$
{has a central character}, say $\chi:Z(\mathfrak{g})\to \mathbb{C}$, where 
$Z(\mathfrak{g})$ is the center of the universal enveloping algebra 
$U(\mathfrak{g})$ of $\mathfrak{g}$. Assume that $L$ is such that 
$\chi$ coincides with the central character of the trivial $\mathfrak{g}$-module.

Then Duflo's Theorem \cite{Du} asserts that there exists $w\in W$ such that
the annihilator $\mathrm{Ann}_{U(\mathfrak{g})}(L)$ of $L$ in $U(\mathfrak{g})$
coincides with $\mathrm{Ann}_{U(\mathfrak{g})}(L_w)$.

Further, a result of Barbasch and Vogan \cite{BV1,BV2} says that,
given $x,y\in W$, we have $\mathrm{Ann}_{U(\mathfrak{g})}(L_x)\subset 
\mathrm{Ann}_{U(\mathfrak{g})}(L_y)$ if and only if $x\geq_L y$.

\subsection{Harish-Chandra bimodules}\label{s3.2}

Recall, see \cite[Kapitel~6]{Ja}, that a  $\mathfrak{g}$-$\mathfrak{g}$-bimodule
is called a Harish-Chandra bimodule if it is finitely generated and, additionally,
if the adjoint action of $\mathfrak{g}$ on it is locally finite and has finite multiplicities.

A typical example of a Harish-Chandra bimodule is the bimodule
$U(\mathfrak{g})/\mathrm{Ann}_{U(\mathfrak{g})}(M)$, for $M\in \mathcal{O}_0$.

Here is another example: let $M,N$ be objects in $\mathcal{O}_0$, then
the space $\mathcal{L}(M,N)$ of linear maps from $M$ to $N$ on which the adjoint 
action of $\mathfrak{g}$ is locally finite is a Harish-Chandra bimodule.
If $M=N$, then $U(\mathfrak{g})/\mathrm{Ann}_{U(\mathfrak{g})}(M)\subset 
\mathcal{L}(M,M)$.

\subsection{Classical Kostant's problem}\label{s3.3}

Kostant's problem, as advertised in \cite{Jo}, is formulated as follows:

{\bf Kostant's  Problem.} For which $w\in W$ is the embedding 
$$U(\mathfrak{g})/\mathrm{Ann}_{U(\mathfrak{g})}(L_w)\hookrightarrow 
\mathcal{L}(L_w,L_w)$$ an isomorphism?

We will denote by $\mathbf{K}(w)=\mathbf{K}_\mathfrak{g}(w)
\in\{\mathtt{true},\mathtt{false}\}$ the logical value of the claim
``the embedding  $U(\mathfrak{g})/\mathrm{Ann}_{U(\mathfrak{g})}(L_w)\hookrightarrow 
\mathcal{L}(L_w,L_w)$ is an isomorphism''. 

\subsection{K{\aa}hrstr{\"o}m's conjecture}\label{s3.4}

The following conjecture is due to Johan K{\aa}hrstr{\"o}m, see \cite[Conjecture~1.2]{KMM}.

\begin{conjecture}\label{conj2}
Assume that we are in type $A_{n-1}$, so $W\cong S_n$, the symmetric group.
Let $d\in S_n$ be an involution. Then the following assertions are equivalent:

\begin{enumerate}
\item\label{conj2.1} $\mathbf{K}(d)=\mathtt{true}$.
\item\label{conj2.2} For any $x,y\in W$ such that $x\neq y$ and $\theta_x L_d\neq 0$ and
$\theta_y L_d\neq 0$, we have $\theta_x L_d\not\cong \theta_y L_d$.
\item\label{conj2.3} For any $x,y\in W$ such that $x\neq y$ and $\theta_x L_d\neq 0$ and
$\theta_y L_d\neq 0$, we have $[\theta_x L_d]\neq [\theta_y L_d]$
in $\mathbf{Gr}[\mathcal{O}_0^\mathbb{Z}]$.
\item\label{conj2.4} For any $x,y\in W$ such that $x\neq y$ and $\theta_x L_d\neq 0$ and
$\theta_y L_d\neq 0$, we have $[\theta_x L_d]\neq [\theta_y L_d]$
in $\mathbf{Gr}[\mathcal{O}_0]$.
\end{enumerate} 
\end{conjecture}

\subsection{Kostant's problem vs K{\aa}hrstr{\"o}m's conjecture
and the indecomposability conjecture}\label{s3.5}

One of the main results of \cite{KMM} is \cite[Theorem~B]{KMM}, which asserts that
Conjecture~\ref{conj2}\eqref{conj2.1} is equivalent to the conjunction of 
Conjecture~\ref{conj2}\eqref{conj2.2} with $\mathbf{KM}(*,d)$.

\subsection{Known results on Kostant's problem}\label{s3.6}

Here is a (possibly incomplete) list of known results on Kostant's problem,
in particular, a  list of the special cases in which 
the answer to Kostant's problem is known.

\begin{itemize}
\item In type $A$, the value $\mathbf{K}(w)$ is constant on Kazhdan--Lusztig left cells,
see \cite[Theorem~61]{MS}.
\item Let $\mathfrak{p}$ be a parabolic subalgebra of $\mathfrak{g}$
containing $\mathfrak{h}\oplus \mathfrak{n}_+$. Let $W^{\mathfrak{p}}$
be the corresponding parabolic subgroup of $W$ and $w_0^\mathfrak{p}$
be the longest element in $W^\mathfrak{p}$. Then $\mathbf{K}(w_0^\mathfrak{p}w_0)=
\mathtt{true}$, see \cite[Theorem~4.4]{GJ} and \cite[Section~7.32]{Ja}.
\item Let $\mathfrak{p}$ be a parabolic subalgebra of $\mathfrak{g}$
and $s\in W^{\mathfrak{p}}$ a simple reflection.
Then $\mathbf{K}(sw_0^\mathfrak{p}w_0)=
\mathtt{true}$, see \cite[Theorem~1]{Ma}.
\item Let $\mathfrak{p}$ be a parabolic subalgebra of $\mathfrak{g}$ and $w\in W^{\mathfrak{p}}$.
Let $\mathfrak{a}$ be the semi-simple part of the Levi quotient of $\mathfrak{p}$. 
Then we have $\mathbf{K}_{\mathfrak{a}}(w)=\mathtt{true}$ if and only if 
$\mathbf{K}_{\mathfrak{g}}(ww_0^\mathfrak{p}w_0)=\mathtt{true}$, see \cite[Theorem~1.1]{Ka}.
\item \cite[Theorem~1]{KaM} gives a module-theoretic characterization of
the statement $\mathbf{K}(w)=\mathtt{true}$.
\item A complete answer to Kostant's problem for $\mathfrak{sl}_n$, where $n=2,3,4,5$,
is given in \cite[Section~4]{KaM}. There one can also find a partial answer
for $\mathfrak{sl}_6$. Some further cases for $\mathfrak{sl}_6$ are
dealt with in \cite[Section~6]{Ka}. A complete answer to Kostant's problem for $\mathfrak{sl}_6$
is given in \cite[Subsection~10.1]{KMM}.
\end{itemize}

\section{Fully commutative permutations and the Temperley-Lieb algebra}\label{s4}

\subsection{Fully commutative permutations}\label{s4.1}

Consider the symmetric group $S_n$ as a Coxeter group in the usual way,
that is, fixing the elementary transpositions $s_i:=(i,i+1)$, for $i=1,2,\dots,n-1$,
as the set of simple reflections. The classical presentation of $S_n$
with respect to these generators has the following relations:
\begin{eqnarray}
s_i^2&=&e;\label{pres1}\\
s_is_j&=&s_js_i,\text{ if } |i-j|\neq 1;\label{pres2}\\
s_is_{i\pm 1}s_i&=&s_{i\pm 1}s_is_{i\pm 1}.\label{pres3}
\end{eqnarray}

The classical result of Matsumoto \cite{Mat} asserts that, for $w\in W$, 
any reduced expression  of $w$ can be transformed into any other 
reduced expression of $w$ by using only the relations in \eqref{pres2} and \eqref{pres3}.
An element $w\in S_n$ is said to be {\em fully commutative} provided that any
reduced expression of $w$ can be transformed into any other reduced expression
of $w$ by only using relation \eqref{pres2}.

The classical Robinson--Schensted correspondence, see \cite{Sc} and \cite[Section~3.1]{Sa},
\begin{displaymath}
\mathbf{RS}:S_n\to \coprod_{\lambda\vdash n}
\mathbf{SYT}_\lambda\times \mathbf{SYT}_\lambda,
\end{displaymath}
assigns to an element of $S_n$ a pair of standard Young tableaux of the same shape
(that shape is a partition of $n$). An element $w\in S_n$ is fully commutative
provided that the shape of the tableaux in $\mathbf{RS}(w)$ is a partition with
at most two parts. Alternatively, an element $w$ is fully commutative provided that
it avoids the $(3,2,1)$-pattern, that is, the longest decreasing subsequence for $w$
has length at most two. We refer to \cite{Fa95,Fa96} for details.

\subsection{Temperley-Lieb algebra}\label{s4.2}

Let $\mathbb{A}$ be a commutative ring and $\delta\in\mathbb{A}$. For $n\geq 1$,
the corresponding Temperley-Lieb algebra $\mathbf{TL}_n(\mathbb{A},\delta)$
is an $\mathbb{A}$-algebra which is free as an $\mathbb{A}$-module with a
basis consisting of all planar (i.e. non-crossing) pairings of $2n$ points
in a plane. Composition is defined via concatenation of diagrams,
straightening the outcome to a new diagram and, finally, multiplication 
with $\delta^k$, where $k$ is the number of closed loops removed
during the straightening procedure. We refer to \cite[Chapter~6]{Mar} for details.
Here is an example:
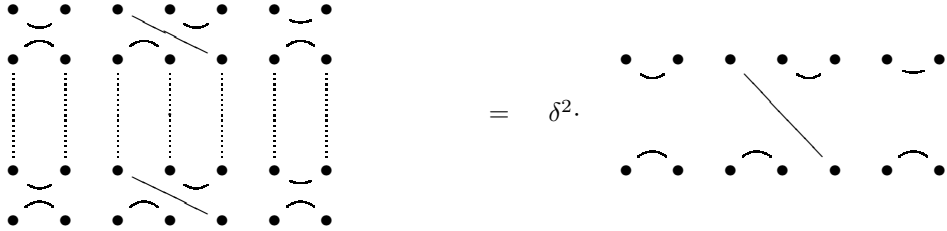
\begin{displaymath}
\xymatrix@R=3mm@C=3mm{
\bullet\ar@{-}@/_7pt/[r]&\bullet&\bullet\ar@{-}[drr]&\bullet\ar@{-}@/_7pt/[r]&
\bullet&\bullet\ar@{-}@/_5pt/[r]&\bullet&&&&&&
&&&&&&\\
\bullet\ar@{-}@/^7pt/[r]\ar@{.}[dd]&\bullet\ar@{.}[dd]&\bullet\ar@{-}@/^7pt/[r]\ar@{.}[dd]
&\bullet\ar@{.}[dd]&\bullet\ar@{.}[dd]&\bullet\ar@{-}@/^7pt/[r]\ar@{.}[dd]&\bullet\ar@{.}[dd]&&&&&&
\bullet\ar@{-}@/_7pt/[r]&\bullet&\bullet\ar@{-}[ddrr]&\bullet\ar@{-}@/_7pt/[r]&
\bullet&\bullet\ar@{-}@/_5pt/[r]&\bullet\\
&&&\cdot &&&&&&&=&\delta^2\cdot &
&&&&&&\\
\bullet\ar@{-}@/_7pt/[r]&\bullet&\bullet\ar@{-}[drr]&\bullet\ar@{-}@/_7pt/[r]&
\bullet&\bullet\ar@{-}@/_5pt/[r]&\bullet&&&&&&
\bullet\ar@{-}@/^7pt/[r]&\bullet&\bullet\ar@{-}@/^7pt/[r]
&\bullet&\bullet&\bullet\ar@{-}@/^7pt/[r]&\bullet\\
\bullet\ar@{-}@/^7pt/[r]&\bullet&\bullet\ar@{-}@/^7pt/[r]
&\bullet&\bullet&\bullet\ar@{-}@/^7pt/[r]&\bullet&&&&&&
&&&&&&\\
}
\end{displaymath}

For $\mathbb{A}=\mathbb{Z}[v,v^{-1}]$ and $\delta=v+v^{-1}$, 
one can give the following alternative description of $\mathbf{TL}_n(\mathbb{A},\delta)$.
It is the quotient of the Hecke algebra $\mathbf{H}_n$ for $S_n$ modulo
the ideal generated by all $\underline{H}_w$, where $w$ is not   
fully commutative. We note that the latter ideal is just 
the $\mathbb{Z}[v,v^{-1}]$-span of all such $\underline{H}_w$. 
By construction, $\mathbf{TL}_n(\mathbb{Z}[v,v^{-1}],v+v^{-1})$
has a basis given by all $\underline{H}_w$, for $w$  fully
commutative. Setting $\mathtt{e}_i:=\underline{H}_{s_i}$, for $i=1,2,\dots,n-1$, we have
the following presentation for $\mathbf{TL}_n(\mathbb{Z}[v,v^{-1}],v+v^{-1})$:
\begin{eqnarray}
\mathtt{e}_i^2&=&{(v+v^{-1})}\mathtt{e}_i;\label{tlpres1}\\
\mathtt{e}_i\mathtt{e}_j&=&\mathtt{e}_j\mathtt{e}_i,\text{ if } |i-j|\neq 1;\label{tlpres2}\\
\mathtt{e}_i\mathtt{e}_{i\pm 1}\mathtt{e}_i&=&\mathtt{e}_i.\label{tlpres3}
\end{eqnarray}
Here $\mathtt{e}_i$ corresponds to the diagram
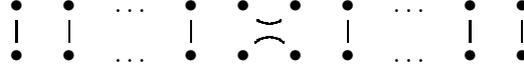
\begin{displaymath}
\xymatrix@R=3mm@C=3mm{
\bullet\ar@{-}[d]&\bullet\ar@{-}[d]&\dots &\bullet\ar@{-}[d]
&\bullet\ar@{-}@/_7pt/[r]&\bullet&\bullet\ar@{-}[d]&\dots &\bullet\ar@{-}[d]&\bullet\ar@{-}[d]\\
\bullet&\bullet&\dots &\bullet&\bullet\ar@{-}@/^7pt/[r]&\bullet&\bullet&\dots &\bullet&\bullet
}
\end{displaymath}
where the horizontally paired points are $i$ and $i+1$.

For a fully commutative element $w\in W$, we denote by $\mathtt{e}_w$
the image of $\underline{H}_w$ in $\mathbf{TL}_n(\mathbb{Z}[v,v^{-1}],v+v^{-1})$.
Taking any reduced expression $w=s_1s_2\dots s_k$, we have
$\mathtt{e}_w=\mathtt{e}_{s_1}\mathtt{e}_{s_2}\dots \mathtt{e}_{s_k}$,
see \cite{Fa95}.

The number of cups (or caps) in $\mathtt{e}_w$ coincides with the value
of Lusztig's $\mathbf{a}$-function on $w$. Also, the cups and caps determine
the corresponding Kazhdan--Lusztig right and left pre-orders, respectively.
For example, $x\leq_L y$ if and only each cap of $\mathtt{e}_x$
is a cap of $\mathtt{e}_y$. {For details, see~\cite[Section 3]{FG97}.}

\begin{remark}\label{rem4.2-1}
\emph{
Let $x,y,z\in S_n$ be fully commutative elements such that 
$\mathtt{e}_x\mathtt{e}_y=f(v)\mathtt{e}_z$, for some $f(v)\in\mathbb{Z}[v,v^{-1}]$. 
Then the number of cups (resp. caps) in $\mathtt{e}_z$
is greater than or equal to the maximum of 
the numbers of cups (resp. caps) in $\mathtt{e}_x$ and $\mathtt{e}_y$.
}
\end{remark}

\begin{remark}\label{rem4.2-2}
\emph{
Let $x\in S_n$ be a fully commutative element. Then $\mathtt{e}_{x^{-1}}$ 
is obtained from $\mathtt{e}_x$ by reflection in a horizontal line.
}
\end{remark}

\subsection{Results of Brundan and Stroppel}\label{s4.3}

Given two arbitrary fully commutative elements $x,y\in S_n$, 
the corresponding module $\theta_xL_y$ in {the category 
$\mathcal{O}$} for $\mathfrak{sl}_n$ is either indecomposable or zero.
If $y$ belongs to a Kazhdan--Lusztig right
cell containing an element of the form $w_0^{\mathfrak{p}}w_0$,
for some parabolic $\mathfrak{p}$, this follows by
combining \cite[Theorem~4.11]{BS} and \cite[Theorem~1.1]{BS2}.
{ Indeed, \cite[Theorem~1.1]{BS2} asserts, among other things, that,
in type $A$, the principal block of the parabolic category $\mathcal{O}$ 
corresponding to a maximal parabolic subalgebra is equivalent to 
the module category of a certain diagrammatic algebra studied in 
\cite{BS}. For the latter algebra, the claim of 
\cite[Theorem~4.11(iii)]{BS} translates exactly into the statement 
that $\theta_xL_y$ is either zero or has simple top.}
Using \cite[Proposition~35]{MS}, we can remove this assumption on the right cell of $y$.
In particular, this implies that $\mathbf{KM}(*,w)=\mathtt{true}$ for
any fully commutative $w\in S_n$.

Note that, by \cite[Theorem~5.1]{Ge}, the two-sided order on $S_n$ is given by the 
dominance order on partitions. Consequently, if $x\in S_n$ is not 
fully commutative while $y\in S_n$ is, then $\theta_xL_y=0$ by
\cite[Lemma~12]{MM1}. Combined with \cite[Theorem~B]{KMM}, we
obtain equivalence of assertions \eqref{conj2.1} and \eqref{conj2.2}
in Conjecture~\ref{conj2} in the case of fully commutative 
elements $x,y\in S_n$.

% 
% 
% \subsection{Left stabilizers in Temperley-Lieb semigroups}\label{s4.4}
% 
% 

\section{Main result}\label{s5}

\subsection{Special fully commutative involutions}\label{s5.1}

For $i\in\{1,2,\dots,n-1\}$, set $\sigma_{i,0}:=s_i$.
Also, for  $j\in\{1,\dots,\min(i-1,n-1-i)\}$, consider the
following element of $S_n$:
\begin{displaymath}
\sigma_{i,j}=
s_i(s_{i-1}s_{i+1})(s_{i-2}s_is_{i+2})\cdots
(s_{i-j}s_{i-j+2}\cdots s_{i+j})\cdots(s_{i-2}s_is_{i+2})(s_{i-1}s_{i+1})s_i.
\end{displaymath}
As a permutation, the element $\sigma_{i,j}$ has the following diagram:

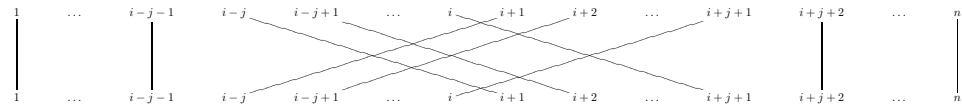
\resizebox{\textwidth}{!}{
$
\xymatrix@R=20mm@C=12mm{
1\ar@{-}[d]&\dots&i-j-1\ar@{-}[d]&i-j\ar@{-}[drrrr]&i-j+1\ar@{-}[drrrr]&\dots
&i\ar@{-}[drrrr]&i+1\ar@{-}[dllll]&i+2\ar@{-}[dllll]&\dots&i+j+1\ar@{-}[dllll]&
i+j+2\ar@{-}[d]&\dots&n\ar@{-}[d]\\
1&\dots&i-j-1&i-j&i-j+1&\dots&i&i+1&i+2&\dots&i+j+1&i+j+2&\dots&n
}
$
}

It is readily seen that this element is $(3,2,1)$-avoiding and hence
fully commutative (cf. \cite[Theorem~2.1]{BJS}). The corresponding Temperley-Lieb diagram is:

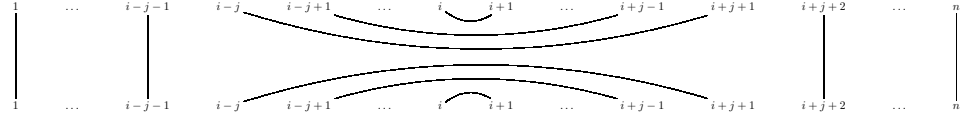
\resizebox{\textwidth}{!}{
$
\xymatrix@R=25mm@C=12mm{
1\ar@{-}[d]&\dots&i-j-1\ar@{-}[d]&i-j\ar@/_3pc/@{-}[rrrrrrr]&i-j+1\ar@/_2pc/@{-}[rrrrr]&\dots
&i\ar@/_1pc/@{-}[r]&i+1&\dots&i+j-1&i+j+1&
i+j+2\ar@{-}[d]&\dots&n\ar@{-}[d]\\
1&\dots&i-j-1&i-j\ar@/^3pc/@{-}[rrrrrrr]&i-j+1\ar@/^2pc/@{-}[rrrrr]
&\dots&i\ar@/^1pc/@{-}[r]&i+1&\dots&i+j-1&i+j+1&i+j+2&\dots&n
}
$
}

We will call the set $\{i-j,i-j+1,\dots,i+j+1\}$ the {\em support} of $\sigma_{i,j}$
and the set $\{i-j-1,i-j,i-j+1,\dots,i+j+1,i+j+2\}$ the {\em extended 
support} of $\sigma_{i,j}$. Elements of the form $\sigma_{i,j}$ will be called special.

We will say that $\sigma_{i,j}$ and $\sigma_{i',j'}$ 
are {\em distant} provided that their extended supports intersect in at most one element.
For example, $s_1=\sigma_{1,0}$ and  $s_4=\sigma_{4,0}$ are distant
since $\{0,1,2,3\}\cap\{3,4,5,6\}=\{3\}$, while 
$s_1=\sigma_{1,0}$ and  $s_3=\sigma_{3,0}$ are not distant
since $\{0,1,2,3\}\cap\{2,3,4,5\}=\{2,3\}$ is neither empty nor a singleton. 
We note that, in the  latter example, the supports of $\sigma_{1,0}$ and 
$\sigma_{3,0}$ do not have common elements.

\subsection{Formulation of the main result}\label{s5.2}

\begin{theorem}\label{thm5.2.1}
Conjecture~\ref{conj2} is true for all fully commutative involutions
in $S_n$. Moreover, if $d\in S_n$ is a fully commutative involution,
then $\mathbf{K}(d)=\mathtt{true}$ if and only if $d$ is a product
of pairwise distant special elements.
\end{theorem}

In other words, for a fully commutative involution $d$, we have
$\mathbf{K}(d)=\mathtt{true}$ if and only if any two  
non-nested cups (or caps) in the Temperley-Lieb diagram of $d$
are separated by at least one vertical line. We will call such fully commutative
involutions {\em Kostant involutions}.

\begin{corollary}\label{cor5.2.1-2}
If $w\in S_n$ is a fully commutative element, then
$\mathbf{K}(w)=\mathtt{true}$ if and only if any two caps in 
$\mathtt{e}_w$ are either nested or separated by at least one propagating line.
\end{corollary}

\begin{proof}
This follows from Theorem~\ref{thm5.2.1}, the first bullet in
Subsection~\ref{s3.6}, the paragraph before Remark~\ref{rem4.2-1}, 
{the last paragraph of Section \ref{s2.6}, 
as well as the fact that Lusztig's $\mathbf{a}$-function is constant on left cells.}
\end{proof}

The two directions of Theorem~\ref{thm5.2.1} will be proved in 
Subsections~\ref{s5.3} and \ref{s5.4}.

\subsection{Auxiliary lemmata}\label{s5.25}

\begin{lemma}\label{l5.25-1}
Let $x$ and $y$ be two fully commutative elements in $S_n$.
Then $\theta_x L_y\neq 0$ is equivalent to the condition
that { each cup of $\mathtt{e}_{x^{-1}}$ is a cup of  $\mathtt{e}_y$.}
\end{lemma}

\begin{proof}
{ By \cite[Lemma~12]{MM1}, the condition $\theta_x L_y\ne 0$ is equivalent to 
the condition $x^{-1} \leq_L y$. Note that $(\theta_x)^* \cong \theta_{x^{-1}}$
and that the diagram for $x^{-1}$ is the flip of the diagram for $x$. 
As recalled in the text above Remark~\ref{rem4.2-1}, the condition $x^{-1} \leq_L y$ holds if and only if the cups of $e_{x^{-1}}$ form a subset of those of $e_y$. By Remark~\ref{rem4.2-1}, this implies the result.
}
\end{proof}

\begin{lemma}\label{l5.25-2}
Let $x,y,z$ be three fully commutative elements in $S_n$ such that
\begin{displaymath}
\mathtt{e}_x\mathtt{e}_y=f(v)\mathtt{e}_z,\qquad
\text{ for some } f(v)\in{\mathbb{Z}[v,v^{-1}]}. 
\end{displaymath}
Then $\theta_x \theta_y $ is isomorphic to $\theta_z^{f(1)}\oplus \theta$,
where $\theta L_w=0$ for any fully commutative element $w\in S_n$.
\end{lemma}

\begin{proof}
This follows from the realization of the Temperley-Lieb algebra as a
quotient of the Hecke algebra and the action of the latter on $\mathcal{O}_0$
via projective functors. { Recall from Section~\ref{s2.6} that, 
for $w \in \{x,y,z\}$, the Grothendieck class 
$[\theta_w]\in \mathbf{Gr}_\oplus[\cP^\mathbb{Z}]$ corresponds to the Kazhdan-Lusztig basis element 
$\underline{H}_w\in \mathbf{H}$, which, in turn, is mapped to the Temperley-Lieb diagram 
$e_w \in \mathbf{TL}_n(\mathbb{Z}[v,v^{-1}],v + v^{-1})$. Therefore, the 
combinatorics of the $\theta_w$ corresponds precisely to the combinatorics of the $e_w$ 
(after putting $v=1$). All this is, of course, up to higher order Kazhdan-Lusztig basis 
elements. These latter elements correspond to the additional summand $\theta$ in the formulation
and are killed in the Temperley-Lieb  quotient. In particular, we also have  $\theta L_w=0$,
for all fully commutative $w$.}
\end{proof}

\subsection{Positive answer}\label{s5.3}

Let $d$ be a fully commutative involution which is a product of 
pairwise distant special elements. We are going to prove that $d$
has the property described in Conjecture~\ref{conj2}\eqref{conj2.4}.

Recall that, for fully commutative elements, assertions \eqref{conj2.1}
and \eqref{conj2.2} in Conjecture~\ref{conj2} are equivalent. We also have the
obvious implications \eqref{conj2.4}$\Rightarrow$\eqref{conj2.3}$\Rightarrow$\eqref{conj2.2}.
Hence $\mathbf{K}(d)=\mathtt{true}$ for any fully commutative involution $d$ that
is a product of  pairwise distant special elements.

Let $x,y\in  S_n$ be two different elements
such that $\theta_x L_d\neq 0$ and $\theta_y L_d\neq 0$. {In particular, 
this implies that $x$ and $y$ are fully commutative, see Section~\ref{s4.3}. }
To prove that $d$ has the property described in 
Conjecture~\ref{conj2}\eqref{conj2.4}, it is enough to find some
$u,v\in S_n$ such that $\dim \mathrm{Hom}_{\mathfrak{g}}(\theta_uP_v,\theta_x L_d)\neq 
\dim \mathrm{Hom}_{\mathfrak{g}}(\theta_u P_v,\theta_y L_d)$. By adjunction,
we have
\begin{displaymath}
\begin{array}{rcl}
\mathrm{Hom}_{\mathfrak{g}}(\theta_uP_v,\theta_x L_d)&{\cong}&
\mathrm{Hom}_{\mathfrak{g}}(\theta_{x^{-1}} \theta_uP_v,L_d),\\
\mathrm{Hom}_{\mathfrak{g}}(\theta_uP_v,\theta_y L_d)&{\cong}&
\mathrm{Hom}_{\mathfrak{g}}(\theta_{y^{-1}} \theta_uP_v,L_d).
\end{array}
\end{displaymath}
Note that, for a projective module $P$, the dimension of
$\mathrm{Hom}_{\mathfrak{g}}(P,L_d)$ equals the multiplicity of
$P_d$ as a direct summand of $P$. Therefore, we need to 
find $u$ and $v$ such that the multiplicity of $\theta_d$ as a summand of 
$\theta_{x^{-1}}\theta_u\theta_v$ is different from the multiplicity of 
$\theta_d$ as a summand of $\theta_{y^{-1}}\theta_u\theta_v$.

Since $\theta_d$ is self-adjoint and both $u$ and $v$ are arbitrary, we can 
reformulate this as follows: find $u$ and $v$ such that the multiplicity 
of $\theta_d$ as a summand of $\theta_v\theta_{u}\theta_x$ differs
from the multiplicity of $\theta_d$ as a summand of $\theta_v\theta_{u}\theta_y$.

{
For future use, we record the following technical lemma.

\begin{lemma}\label{multi}
In the notation from above, the multiplicity of $\theta_d$ as a summand of 
$\theta_d\theta_{x^{-1}}\theta_x$ equals $2^{2a}$.
\end{lemma}

\begin{proof}
The diagram $\mathtt{e}_{x^{-1}}\mathtt{e}_x$ is self-dual
(i.e. symmetric with respect to reflection in a horizontal line), has $a$ cups
and $a$ caps, and also $a$ circles in the middle before straightening.
By Lemma~\ref{l5.25-1}, each cup of this diagram corresponds to a cap of $\mathtt{e}_d$.
Therefore, the number of circles removed when
straightening the product $\mathtt{e}_d\mathtt{e}_{x^{-1}}\mathtt{e}_x$
equals $2a$. This implies the claim.
\end{proof}
}

Without loss of generality, we may assume that the number  $a$ of caps in
$\mathtt{e}_x$ is greater than or equal to the number $b$ of caps in
$\mathtt{e}_y$. 

In most cases below, we will see that the choice $v=d$ and $u=x^{-1}$ does the job.

{\bf Case~1.} Let us assume that $a>b$, so there is at least one cup of $\mathtt{e}_x$
which is not a cup of $\mathtt{e}_y$. Set $u=x^{-1}$  and $v=d$.

{ By Lemma \ref{multi}, the multiplicity of $\theta_d$ as a summand of 
$\theta_d\theta_{x^{-1}}\theta_x$ equals $2^{2a}$.}

If the underlying diagram of $\mathtt{e}_d\mathtt{e}_{x^{-1}}\mathtt{e}_y$
is not $\mathtt{e}_d$, then, by Lemma~\ref{l5.25-2}, the multiplicity of $\theta_d$ as a summand of 
$\theta_d\theta_{x^{-1}}\theta_y$ equals $0$ and we are done.
If the underlying diagram of $\mathtt{e}_d\mathtt{e}_{x^{-1}}\mathtt{e}_y$
is $\mathtt{e}_d$, we need to compute the number of closed loops 
removed when straightening the product 
$\mathtt{e}_d\mathtt{e}_{x^{-1}}\mathtt{e}_y$. Each closed loop contains at least one
cap. The total number of original caps in $\mathtt{e}_d\mathtt{e}_{x^{-1}}\mathtt{e}_y$
before straightening is $a+b+c$, where $c$ is the number of caps in $\mathtt{e}_d$. As 
$c\geq a> b$ by Lemma~\ref{l5.25-1}, the resulting diagram has at least $c$ caps,
see Remark~\ref{rem4.2-1}. Therefore, 
the number of closed loops removed during the straightening procedure is at most
$a+b<2a$. Consequently, the multiplicity of $\theta_d$ as a summand of 
$\theta_d\theta_{x^{-1}}\theta_y$ is {at most} $2^{a+b}<2^{2a}$. This completes Case~1.

{\bf Case~2.} Let us assume that $a=b$ and $\mathtt{e}_x$ and $\mathtt{e}_y$
have exactly the same cups. Set $u=x^{-1}$  and $v=d$.

Since $x$ and $y$ are different by assumption, 
there should be at least one cap of $\mathtt{e}_x$ which is not a cap of
$\mathtt{e}_y$. This implies that the underlying diagrams $\mathtt{e}_q$ of 
$\mathtt{e}_{x^{-1}}\mathtt{e}_x$ and $\mathtt{e}_p$ of $\mathtt{e}_{x^{-1}}\mathtt{e}_y$
are different. 

Now we claim that the underlying diagram for $\mathtt{e}_d\mathtt{e}_{x^{-1}}\mathtt{e}_x$
is $\mathtt{e}_d$, while the underlying diagram for $\mathtt{e}_d\mathtt{e}_{x^{-1}}\mathtt{e}_y$
is different from $\mathtt{e}_d$. This, of course, will complete 
the present case. The underlying diagram of $\mathtt{e}_d\mathtt{e}_{x^{-1}}\mathtt{e}_x$
being $\mathtt{e}_d$ follows directly from the fact that
each cap of $\mathtt{e}_x$ is also a cap of $\mathtt{e}_d$, see Lemma~\ref{l5.25-2}.
Here is an example:
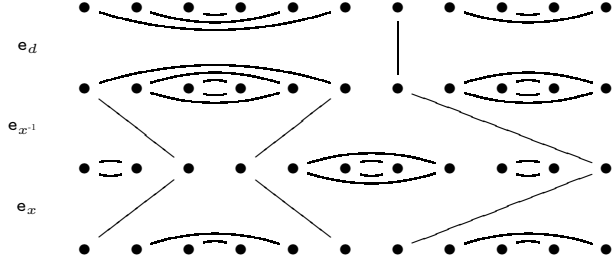
\begin{displaymath}
\xymatrix@C=3mm@R=7mm{ 
\bullet\ar@/_5mm/@{{}{ }{}}[d]_{\mathtt{e}_d}\ar@/_3mm/@{-}[rrrrr]&
\bullet\ar@/_2mm/@{-}[rrr]&\bullet\ar@/_1mm/@{-}[r]&\bullet
&\bullet&\bullet&\bullet\ar@{-}[d]&\bullet\ar@/_2mm/@{-}[rrr]&\bullet\ar@/_1mm/@{-}[r]&\bullet&\bullet\\
\bullet\ar@{-}[drr]\ar@/^3mm/@{-}[rrrrr]\ar@/_5mm/@{{}{ }{}}[d]_{\mathtt{e}_{x^{\resizebox{1mm}{!}{-$1$}}}}
&\bullet\ar@/^2mm/@{-}[rrr]\ar@/_2mm/@{-}[rrr]
&\bullet\ar@/^1mm/@{-}[r]\ar@/_1mm/@{-}[r]&\bullet&\bullet&
\bullet\ar@{-}[dll]&\bullet\ar@{-}[drrrr]&\bullet\ar@/^2mm/@{-}[rrr]\ar@/_2mm/@{-}[rrr]
&\bullet\ar@/^1mm/@{-}[r]\ar@/_1mm/@{-}[r]&\bullet&\bullet\\
\bullet\ar@/_1mm/@{-}[r]\ar@/^1mm/@{-}[r]\ar@/_5mm/@{{}{ }{}}[d]_{\mathtt{e}_x}&\bullet&\bullet&\bullet&
\bullet\ar@/_2mm/@{-}[rrr]\ar@/^2mm/@{-}[rrr]&\bullet\ar@/_1mm/@{-}[r]\ar@/^1mm/@{-}[r]&\bullet
&\bullet&\bullet\ar@/_1mm/@{-}[r]\ar@/^1mm/@{-}[r]&\bullet&\bullet\\
\bullet\ar@{-}[urr]&\bullet\ar@/^2mm/@{-}[rrr]&\bullet\ar@/^1mm/@{-}[r]
&\bullet&\bullet&\bullet\ar@{-}[ull]&\bullet\ar@{-}[urrrr]
&\bullet\ar@/^2mm/@{-}[rrr]&\bullet\ar@/^1mm/@{-}[r]&\bullet&\bullet\\
}
\end{displaymath}

Let us now look at the underlying diagram for $\mathtt{e}_d\mathtt{e}_{x^{-1}}\mathtt{e}_y$.
Recall that:
\begin{itemize}
\item $\mathtt{e}_p$ and $\mathtt{e}_q$ have exactly the same cups;
\item not all caps in $\mathtt{e}_p$ and $\mathtt{e}_q$ are the same;
\item any cap of $\mathtt{e}_p$ and any cap of $\mathtt{e}_q$
is also a cap of $\mathtt{e}_d$;
\item $d=d_1d_2\cdots d_r$ is a product of pairwise distant special elements $d_i$.
\end{itemize}
This implies that there exists a special factor $d_i$ of $d$ such that
the number of caps $\mathtt{e}_{d_i}$ shares with 
$\mathtt{e}_q$ is different from the number of caps $\mathtt{e}_{d_i}$ 
shares with $\mathtt{e}_p$.

Take the leftmost such factor. Due to this assumption,
to the left of this factor in $\mathtt{e}_{d}$, the diagrams 
$\mathtt{e}_d\mathtt{e}_{x^{-1}}\mathtt{e}_x$ and
$\mathtt{e}_d\mathtt{e}_{x^{-1}}\mathtt{e}_y$ fully agree. We have two subcases.

{\bf Subcase~2a.} The number of caps $\mathtt{e}_{d_i}$ shares with
$\mathtt{e}_q$ is smaller than the number of caps $\mathtt{e}_{d_i}$ shares with
$\mathtt{e}_p$.

In this case, the additional caps of $\mathtt{e}_{d_i}$, when multiplied with 
$\mathtt{e}_p$, are {\color{orange}moved to the right}. Here is a fairly generic example:
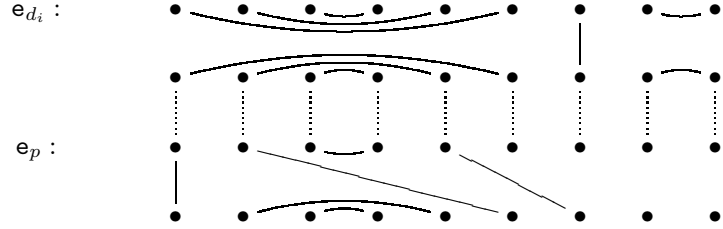
\begin{displaymath}
\xymatrix@C=5mm@R=5mm{
\mathtt{e}_{d_i}:&&\bullet\ar@/_3mm/@{-}[rrrrr]&\bullet\ar@/_2mm/@{-}[rrr]&\bullet\ar@/_1mm/@{-}[r]
&\bullet&\bullet&\bullet&\bullet\ar@{-}[d]&\bullet\ar@/_1mm/@{-}[r]&\bullet\\
&&\bullet\ar@{.}[d]\ar@/^3mm/@{-}[rrrrr]&\bullet\ar@{.}[d]\ar@/^2mm/@{-}[rrr]
&\bullet\ar@{.}[d]\ar@/^1mm/@{-}[r]&\bullet\ar@{.}[d]&
\bullet\ar@{.}[d]&\bullet\ar@{.}[d]&\bullet\ar@{.}[d]&
\bullet\ar@{.}[d]\ar@/^1mm/@{-}[r]&\bullet\ar@{.}[d]\\
\mathtt{e}_p:&&\bullet\ar@{-}[d]&\bullet\ar@{-}@[orange][drrrr]&\bullet\ar@/_1mm/@{-}[r]&
\bullet&\bullet\ar@{-}@[orange][drr]&\bullet&\bullet&\bullet&\bullet\\
&&\bullet&\bullet\ar@/^2mm/@{-}[rrr]&\bullet\ar@/^1mm/@{-}[r]
&\bullet&\bullet&\bullet&\bullet&\bullet&\bullet\\
}
\end{displaymath}
This makes  the resulting diagram different from $\mathtt{e}_d$, as claimed. 

{\bf Subcase~2b.} The number of caps $\mathtt{e}_{d_i}$ shares with $\mathtt{e}_q$
is greater than the number of caps $\mathtt{e}_{d_i}$ shares with $\mathtt{e}_p$.

In this case, we have some additional points corresponding to propagating lines
to the left of the nested caps in $\mathtt{e}_p$. The rightmost of these 
points either hits a cap in $\mathtt{e}_d$ or it hits a propagating line in $\mathtt{e}_d$.
In the former case, we obtain a {\color{magenta}cap
in $\mathtt{e}_d\mathtt{e}_p$} which is not a cap in $\mathtt{e}_d$.
Here is a fairly generic example:
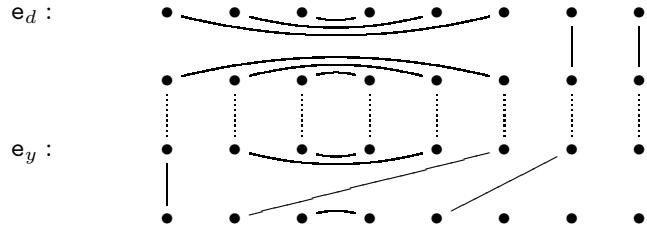
\begin{displaymath}
\xymatrix@C=5mm@R=5mm{
\mathtt{e}_d:&&\bullet\ar@/_3mm/@{-}[rrrrr]&\bullet\ar@/_2mm/@{-}[rrr]&\bullet\ar@/_1mm/@{-}[r]
&\bullet&\bullet&\bullet&\bullet\ar@{-}[d]&\bullet\ar@{-}[d]\\
&&\bullet\ar@{.}[d]\ar@/^3mm/@{-}[rrrrr]&\bullet\ar@{.}[d]\ar@/^2mm/@{-}[rrr]
&\bullet\ar@{.}[d]\ar@/^1mm/@{-}[r]&\bullet\ar@{.}[d]&
\bullet\ar@{.}[d]&\bullet\ar@{.}[d]&\bullet\ar@{.}[d]&
\bullet\ar@{.}[d]\\
\mathtt{e}_y:&&\bullet\ar@{-}@[magenta][d]&\bullet\ar@/_2mm/@{-}[rrr]&\bullet\ar@/_1mm/@{-}[r]&
\bullet&\bullet&\bullet\ar@{-}@[magenta][dllll]&\bullet\ar@{-}[dll]&\bullet\\
&&\bullet&\bullet&\bullet\ar@/^1mm/@{-}[r]&\bullet&\bullet&\bullet&\bullet&\bullet\\
}
\end{displaymath}
This makes the resulting diagram different from $\mathtt{e}_d$.

In the latter case, the resulting diagram has a 
{\color{violet}non-vertical propagating line}.
Here is a fairly generic example:
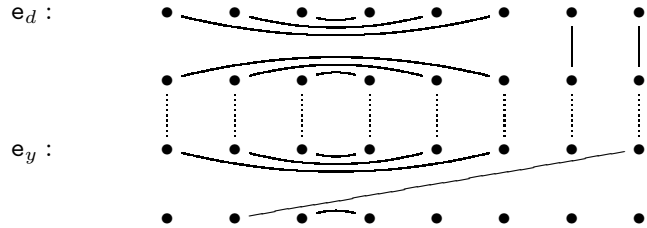
\begin{displaymath}
\xymatrix@C=5mm@R=5mm{
\mathtt{e}_d:&&\bullet\ar@/_3mm/@{-}[rrrrr]&\bullet\ar@/_2mm/@{-}[rrr]&\bullet\ar@/_1mm/@{-}[r]
&\bullet&\bullet&\bullet&\bullet\ar@{-}[d]&\bullet\ar@{-}[d]\\
&&\bullet\ar@{.}[d]\ar@/^3mm/@{-}[rrrrr]&\bullet\ar@{.}[d]\ar@/^2mm/@{-}[rrr]
&\bullet\ar@{.}[d]\ar@/^1mm/@{-}[r]&\bullet\ar@{.}[d]&
\bullet\ar@{.}[d]&\bullet\ar@{.}[d]&\bullet\ar@{.}[d]&
\bullet\ar@{.}[d]\\
\mathtt{e}_y:&&\bullet\ar@/_3mm/@{-}[rrrrr]&\bullet\ar@/_2mm/@{-}[rrr]&\bullet\ar@/_1mm/@{-}[r]&
\bullet&\bullet&\bullet&\bullet&\bullet\ar@{-}@[violet][dllllll]\\
&&\bullet&\bullet&\bullet\ar@/^1mm/@{-}[r]&\bullet&\bullet&\bullet&\bullet&\bullet\\
}
\end{displaymath}
This makes the resulting diagram again different from $\mathtt{e}_d$ and completes Case~2.

{\bf Case~3.} Let us assume that $a=b$ and $\mathtt{e}_x$ and $\mathtt{e}_y$ have 
exactly the same caps. Set $u=x^{-1}$  and $v=d$ as before.

Consider first the situation when the underlying diagrams of 
$\mathtt{e}_{x^{-1}}\mathtt{e}_x$ and $\mathtt{e}_{x^{-1}}\mathtt{e}_y$
coincide,  let us call this diagram $\mathtt{e}_z$. Note that 
$\mathtt{e}_z$ has $a$ caps. As $\mathtt{e}_x$ and $\mathtt{e}_y$
have the same caps but are assumed to be different, not all of their cups
can coincide. In particular, the multiplicity of
$\mathtt{e}_z$ in $\mathtt{e}_{x^{-1}}\mathtt{e}_y$ is strictly smaller
than the multiplicity $2^a$ of $\mathtt{e}_z$
in $\mathtt{e}_{x^{-1}}\mathtt{e}_x$. Consequently, if we now multiply 
with $\mathtt{e}_d$ on the left, we obtain $\mathtt{e}_d$ as the underlying
diagram in both $\mathtt{e}_d\mathtt{e}_{x^{-1}}\mathtt{e}_x$ and 
$\mathtt{e}_d\mathtt{e}_{x^{-1}}\mathtt{e}_y$, but with different multiplicities, 
and we are done with this situation.

Now consider the situation when the underlying diagrams $\mathtt{e}_q$ of 
$\mathtt{e}_{x^{-1}}\mathtt{e}_x$ and $\mathtt{e}_p$ of $\mathtt{e}_{x^{-1}}\mathtt{e}_y$ are different. 
Let us assume, for a contradiction, that 
\begin{displaymath}
\mathtt{e}_d\mathtt{e}_{x^{-1}}\mathtt{e}_x= 
\mathtt{e}_d\mathtt{e}_{x^{-1}}\mathtt{e}_y.
\end{displaymath}
Note that $\mathtt{e}_d\mathtt{e}_{x^{-1}}\mathtt{e}_x=2^{2a}\mathtt{e}_d$
{by Lemma~\ref{multi}}.
The total number of original caps in $\mathtt{e}_d\mathtt{e}_{x^{-1}}\mathtt{e}_y$
is $2a+r$, where $r$ is the number of caps in $\mathtt{e}_d$. 
The underlying diagram of $\mathtt{e}_d\mathtt{e}_{x^{-1}}\mathtt{e}_y$ is
$\mathtt{e}_d$, which accounts for $r$ caps. The only way to get
$2^{2a}$ as the multiplicity of $\mathtt{e}_d$ in $\mathtt{e}_d\mathtt{e}_{x^{-1}}\mathtt{e}_y$
is to have a bijection between the set of 
original caps in $\mathtt{e}_d\mathtt{e}_{x^{-1}}\mathtt{e}_y$ 
and the union of the set of all caps in $\mathtt{e}_d$ with the set of all
closed loops removed during the straightening procedure. In particular,
each such closed loop consists of exactly one cup and one cap.

By the previous paragraph, each propagating line in $\mathtt{e}_d$ must hit 
a propagating line in $\mathtt{e}_p$. Since all 
propagating lines in $\mathtt{e}_d$ are vertical, the same has to 
be true for the corresponding propagating
lines in $\mathtt{e}_p$ {by Lemma~\ref{l5.25-1}}. In other words, the set of propagating lines in
$\mathtt{e}_d$ is a subset of the set of propagating lines in $\mathtt{e}_p$.

Take now a cap $C$ in $\mathtt{e}_d$ which is not nested inside any other cap in 
$\mathtt{e}_d$. Due to our assumptions on $d$, the immediate outside
neighbors of $C$ are propagating lines. Therefore the endpoints of $C$
correspond to either a cup or two propagating lines in $\mathtt{e}_p$.
In the former case, all caps nested inside $C$ correspond to cups of $\mathtt{e}_p$.
In the latter case, both of these propagating lines have
to be vertical, for otherwise there would exist some extra caps in $\mathtt{e}_p$
which are not caps of $\mathtt{e}_d$. 

{
We proceed by induction on the number $\mathbf{k}$ of nested caps contained inside $C$ which do not
correspond to any cups in $\mathtt{e}_p$, to show that 
$\mathtt{e}_p$ and $\mathtt{e}_q$ coincide in the corresponding regions
which interact with $C$ and all inner points of $C$ during the 
multiplications $\mathtt{e}_d\mathtt{e}_p$ and 
$\mathtt{e}_d\mathtt{e}_q$. 
This implies $\mathtt{e}_p=\mathtt{e}_q$, which contradicts our assumption. 
If $\mathbf{k}=0$, the above argument shows that both endpoints of $C$
hit vertical propagating lines in $\mathtt{e}_p$. Since $\mathtt{e}_p$
and $\mathtt{e}_q$ have the same caps and all propagating lines of 
$\mathtt{e}_q$ are vertical by construction, these two propagating lines
of $\mathtt{e}_p$ are also propagating lines of $\mathtt{e}_q$, and
we are done. 
If $\mathbf{k}>0$, there is a unique outermost cap $C'$ nested
inside $C$ by our assumption that $d$ is a product of pairwise
distant special elements. The argument that we just applied to  $C$
applies to $C'$.  Proceeding inductively we obtain that $\mathtt{e}_p$ and $\mathtt{e}_q$ 
coincide at all parts that hit the endpoints and all inner points of $C$}. This completes Case~3.

{\bf Case~4.} Let us assume that $a=b$ and that some of the caps 
and some of the cups in  $\mathtt{e}_x$ and $\mathtt{e}_y$ are different.
Set $u=x^{-1}$  and $v=d$ as before. In most situations,
it is possible to adapt the argument we used in Case~3.

If $\mathtt{e}_d\mathtt{e}_{x^{-1}}\mathtt{e}_x\neq
\mathtt{e}_d\mathtt{e}_{x^{-1}}\mathtt{e}_y$, then we are done.
So, let us assume $\mathtt{e}_d\mathtt{e}_{x^{-1}}\mathtt{e}_x=
\mathtt{e}_d\mathtt{e}_{x^{-1}}\mathtt{e}_y$. Note that 
$\mathtt{e}_d\mathtt{e}_{x^{-1}}\mathtt{e}_x=2^{2a}\mathtt{e}_d$ {by Lemma~\ref{multi}}.

As before, if the underlying diagrams of $\mathtt{e}_{x^{-1}}\mathtt{e}_x$
and $\mathtt{e}_{x^{-1}}\mathtt{e}_y$ are the same, the multiplicity of
$\mathtt{e}_d$ in $\mathtt{e}_d\mathtt{e}_{x^{-1}}\mathtt{e}_y$ is strictly
smaller than $2^{2a}$. Therefore, it remains to consider the situation when 
the underlying diagram $\mathtt{e}_q$ of $\mathtt{e}_{x^{-1}}\mathtt{e}_x$
is different from the underlying diagram $\mathtt{e}_p$ of 
and $\mathtt{e}_{x^{-1}}\mathtt{e}_y$.
If $\mathtt{e}_q$
and $\mathtt{e}_p$ have the same caps, we
can use the argument from Case~3. In particular, we may assume that
not all caps in $\mathtt{e}_q$
and $\mathtt{e}_p$ agree. In this case we will need to construct 
a modification $\mathtt{e}_{d'}$ of $\mathtt{e}_{d}$ which will 
take the place of $\mathtt{e}_v$.

The same argument as in Case~3 shows that the set of propagating lines 
of $\mathtt{e}_d$ is a subset of both the set of propagating lines 
of $\mathtt{e}_p$ and the set of propagating lines 
of $\mathtt{e}_q$. Furthermore, the sets of caps of 
$\mathtt{e}_p$ and $\mathtt{e}_q$ are (different) subsets
of the set of caps of $\mathtt{e}_d$.

Let us consider some full collection $\mathcal{F}$ of nested caps in $\mathtt{e}_d$.
Let $\alpha$ and $\beta$ be the numbers of caps of $\mathtt{e}_p$,
respectively, $\mathtt{e}_q$ contained in this collection.

Note that $\mathtt{e}_q$ has $a$ caps while $\mathtt{e}_p$ has
at least $a$ caps. Therefore, since not all caps in 
$\mathtt{e}_p$ and $\mathtt{e}_q$ 
agree, we can assume that $\alpha>\beta$ for the chosen $\mathcal{F}$.
The argument from Case~3 implies that each cap in $\mathcal{F}$
either hits a cup or two vertical lines in $\mathtt{e}_p$. 
This means that the parts of $\mathtt{e}_p$ and $\mathtt{e}_q$ 
corresponding to $\mathcal{F}$ are as in the following example:

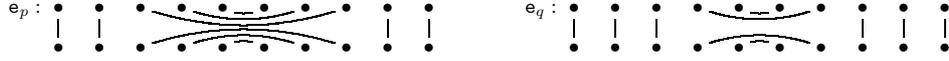
\resizebox{\textwidth}{!}{
$
\mathtt{e}_p:
\xymatrix@C=3mm@R=3mm{
\bullet\ar@{-}[d]&\bullet\ar@{-}[d]&\bullet\ar@/_3mm/@{-}[rrrrr]&\bullet\ar@/_2mm/@{-}[rrr]
&\bullet\ar@/_1mm/@{-}[r]&\bullet&\bullet&\bullet&\bullet\ar@{-}[d]&\bullet\ar@{-}[d]\\
\bullet&\bullet&\bullet\ar@/^3mm/@{-}[rrrrr]&\bullet\ar@/^2mm/@{-}[rrr]&
\bullet\ar@/^1mm/@{-}[r]&\bullet&\bullet&\bullet&\bullet&\bullet\\
}\qquad\qquad
\mathtt{e}_q:
\xymatrix@C=3mm@R=3mm{
\bullet\ar@{-}[d]&\bullet\ar@{-}[d]&\bullet\ar@{-}[d]&\bullet\ar@/_2mm/@{-}[rrr]
&\bullet\ar@/_1mm/@{-}[r]&\bullet&\bullet&\bullet\ar@{-}[d]&\bullet\ar@{-}[d]&\bullet\ar@{-}[d]\\
\bullet&\bullet&\bullet&\bullet\ar@/^2mm/@{-}[rrr]&\bullet\ar@/^1mm/@{-}[r]
&\bullet&\bullet&\bullet&\bullet&\bullet\\
}
$
}

Consider the element $d'$ such that $\mathtt{e}_{d'}$ is the same as
$\mathtt{e}_{d}$ except $\mathcal{F}$ is adjusted as follows:

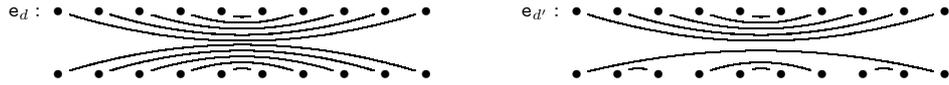
\resizebox{\textwidth}{!}{
$
\mathtt{e}_d:
\xymatrix@C=3mm@R=7mm{
\bullet\ar@/_5mm/@{-}[rrrrrrrrr]&\bullet\ar@/_4mm/@{-}[rrrrrrr]
&\bullet\ar@/_3mm/@{-}[rrrrr]&\bullet\ar@/_2mm/@{-}[rrr]
&\bullet\ar@/_1mm/@{-}[r]&\bullet&\bullet&\bullet&\bullet&\bullet\\
\bullet\ar@/^5mm/@{-}[rrrrrrrrr]&\bullet\ar@/^4mm/@{-}[rrrrrrr]
&\bullet\ar@/^3mm/@{-}[rrrrr]&\bullet\ar@/^2mm/@{-}[rrr]&
\bullet\ar@/^1mm/@{-}[r]&\bullet&\bullet&\bullet&\bullet&\bullet\\
}\qquad\qquad
\mathtt{e}_{d'}:
\xymatrix@C=3mm@R=7mm{
\bullet\ar@/_5mm/@{-}[rrrrrrrrr]&\bullet\ar@/_4mm/@{-}[rrrrrrr]
&\bullet\ar@/_3mm/@{-}[rrrrr]&\bullet\ar@/_2mm/@{-}[rrr]
&\bullet\ar@/_1mm/@{-}[r]&\bullet&\bullet&\bullet&\bullet&\bullet\\
\bullet\ar@/^4mm/@{-}[rrrrrrrrr]&\bullet\ar@/^1mm/@{-}[r]
&\bullet&\bullet\ar@/^2mm/@{-}[rrr]&
\bullet\ar@/^1mm/@{-}[r]&\bullet&\bullet&\bullet\ar@/^1mm/@{-}[r]&\bullet&\bullet\\
}
$
}

Then $\mathtt{e}_{d'}\mathtt{e}_{p}$ is a multiple of $\mathtt{e}_{d}$
while $\mathtt{e}_{d'}\mathtt{e}_{q}$ is not. This completes the proof
for the positive answer. 

\subsection{Negative answer}\label{s5.4}

Let $d$ be a fully commutative involution which is not a product of 
pairwise distant special elements. We are going to prove that $d$
does not have the property described in Conjecture~\ref{conj2}\eqref{conj2.2}.
More explicitly, we will find two different fully commutative elements
$x$ and $y$ such that $\theta_xL_d$ is isomorphic to $\theta_y L_d$.

Recall that, for fully commutative elements, 
assertions \eqref{conj2.1} and \eqref{conj2.2} in Conjecture~\ref{conj2}
are equivalent. We also have the
obvious implications 
$\neg$\eqref{conj2.2}$\Rightarrow\neg$\eqref{conj2.3}$\Rightarrow\neg$\eqref{conj2.4}
of the other assertions.
The above therefore implies  $\mathbf{K}(d)=\mathtt{false}$ and, moreover, the rest of
Conjecture~\ref{conj2} for the involution $d$. 

If $d$ is not a product of pairwise distant special elements, the 
diagram $\mathtt{e}_d$ has two adjacent non-nested caps.
Let us fix a pair $A$ and $B$ 
of such adjacent non-nested caps
(with $A$ on the left). We may assume that they are not nested in 
some other cap which itself has an adjacent non-nested cap. Let $A'$ 
and $B'$ be the corresponding cups.

Define the element $\mathtt{e}_x$ by changing $\mathtt{e}_d$ as follows: 
\begin{itemize}
\item remove $B$ and  $B'$;
\item if applicable, remove all caps in which $A$ and $B$ are nested;
\item remove all cups corresponding to the latter caps;
\item replace all the removed cups and caps by propagating lines.
\end{itemize}
The latter process is unique due to the non-intersection condition.

Define the element $\mathtt{e}_y$  by changing $\mathtt{e}_d$ as follows:
\begin{itemize}
\item remove $A$ and $B'$;
\item if applicable, remove all caps in which $A$ and $B$ are nested;
\item remove all cups corresponding to the latter caps;
\item replace all the removed cups and caps by propagating lines.
\end{itemize}
Again, the latter process is unique due to the 
non-intersection condition. Here is an example, with
{\color{magenta}$A$} and {\color{teal}$B$} colored:

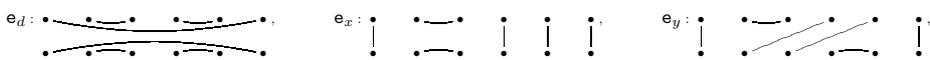
\resizebox{\textwidth}{!}{
$
\resizebox{5mm}{!}{$\mathtt{e}_d$}:\xymatrix@C=7mm@R=5mm{
\bullet\ar@/_3mm/@{-}[rrrrr]&
\bullet\ar@/_1mm/@{-}[r]&\bullet&\bullet\ar@/_1mm/@{-}[r]&\bullet&\bullet\\
\bullet\ar@/^3mm/@{-}[rrrrr]&\bullet\ar@/^1mm/@{-}@[magenta][r]&
\bullet&\bullet\ar@/^1mm/@{-}@[teal][r]&\bullet&\bullet\\
},\qquad\qquad
\resizebox{5mm}{!}{$\mathtt{e}_x$}:\xymatrix@C=7mm@R=5mm{
\bullet\ar@{-}[d]&
\bullet\ar@/_1mm/@{-}[r]&\bullet&\bullet\ar@{-}[d]&\bullet\ar@{-}[d]&\bullet\ar@{-}[d]\\
\bullet&\bullet\ar@/^1mm/@{-}[r]&
\bullet&\bullet&\bullet&\bullet\\
},\qquad\qquad
\resizebox{5mm}{!}{$\mathtt{e}_y$}:\xymatrix@C=7mm@R=5mm{
\bullet\ar@{-}[d]&
\bullet\ar@/_1mm/@{-}[r]&\bullet&\bullet\ar@{-}[dll]&\bullet\ar@{-}[dll]&\bullet\ar@{-}[d]\\
\bullet&\bullet&
\bullet&\bullet\ar@/^1mm/@{-}[r]&\bullet&\bullet\\
},\qquad
$
}

Let $c$ denote the number of caps in $\mathtt{e}_d$ and $a$ the number of caps in 
$\mathtt{e}_x$. Then $a\leq c-1$, by construction, and also $a$ equals the
number of caps in $\mathtt{e}_y$. We are going to prove that $\theta_xL_d$ and 
$\theta_y L_d$ are isomorphic as graded modules. For this we need some 
combinatorial preparation for estimates of graded shifts.
The following statements can probably be deduced from the results of
\cite{BS}, but it is easier to prove them directly.

\begin{lemma}\label{lem5.4.1}
Let $u$ and $w$ be two fully commutative permutations such that 
$\theta_u L_w\neq 0$. Let $k$ be the minimum of the numbers of 
caps in $\mathtt{e}_u$ and in $\mathtt{e}_w$. Then $\theta_u L_w$ is a graded self-dual module
and, for $|i|>k$, the graded component $(\theta_u L_w)_i$ is zero.
\end{lemma}

\begin{proof}
The module  $\theta_u L_w$ is self-dual as it is the image of a self-dual
module $L_w$ under a projective functor. In order to prove the rest of 
the lemma, it is enough to argue that, for $i<-k$, the graded component
$(\theta_u P_w)_i$ is zero. Since the algebra of $\mathcal{O}_0$ is
positively graded, all standard graded 
lifts of projectives live in non-negative degrees.
When computing the product $\mathtt{e}_u\mathtt{e}_w$, the number of closed
loops removed in the straightening procedure is at most $k$.
This gives the scalar $(v+v^{-1})^m$, where $m\leq k$. Therefore
the maximal graded shift of a projective in $\theta_u P_w$
is bounded by $k$. The claim follows.
\end{proof}

\begin{corollary}\label{cor5.4.2}
Both modules $\theta_x L_d$ and $\theta_y L_d$ have simple tops, which 
live in degree $-a$.
\end{corollary}

\begin{proof}
Both modules $\theta_x L_d$ and $\theta_y L_d$ have simple tops by
\cite[Theorem~4.11]{BS}. We prove the second claim for $\theta_x L_d$. 
For $\theta_y L_d$ the arguments are similar. 
Lemma~\ref{lem5.4.1} implies that
the simple top of $\theta_x L_d$ lives in some degree $i\geq -a$. 

Let us now look more closely at the proof of Lemma~\ref{lem5.4.1}.
Note that, when computing the product $\mathtt{e}_x\mathtt{e}_d$,
we need to remove exactly $a$ circles. This means that the
degree $-a$ component of $\theta_x P_d$ is non-zero.

At the same time, any simple subquotient $L_w$ in the radical
of $P_d$ lives in a strictly positive degree. Therefore, {
using Lemma~\ref{lem5.4.1} and an appropriate shift of
grading in the positive direction}, the module $\theta_x L_w$ lives in
degrees that are strictly bigger than $-a$. Thus 
the degree $-a$ component of $\theta_x L_d$ is indeed non-zero.
Due to the positivity of the grading, this component is
the top of $\theta_x L_d$.
\end{proof}

We want to prove that two graded modules $\theta_x L_d$ and $\theta_y L_d$
which have simple tops that live in the same degree are isomorphic.
Due to the positivity of the grading, it is enough to show that there
is a non-zero degree zero morphism from $\theta_x L_d$ to $\theta_y L_d$.
By adjunction, this is equivalent to the existence of a degree zero morphism
from $L_d$ to $\theta_{x^{-1}}\theta_y L_d$.

By construction, $\mathtt{e}_x$ and $\mathtt{e}_y$ have the same 
cups. Therefore, during the straightening of the product 
$\mathtt{e}_{x^{-1}}\mathtt{e}_y$, there are $a$ loops to remove.
Let $\mathtt{e}_p$ be the underlying diagram of $\mathtt{e}_{x^{-1}}\mathtt{e}_y$
and note that it has exactly the same caps as $\mathtt{e}_y$.
By an analogue of Corollary~\ref{cor5.4.2} for $\mathtt{e}_p$
(which works verbatim since $\mathtt{e}_p$ and $\mathtt{e}_y$
have the same caps), the module $\theta_p L_d$ is a self-dual
indecomposable module with simple top in degree $-a$. 
Since $\theta_p$ appears exactly once with shift $a$ in the 
decomposition of $\theta_{x^{-1}}\theta_y$, it follows that 
$L_d$ appears, as a graded module, in the socle of $\theta_{x^{-1}}\theta_y L_d$.
This completes the proof of the negative answer and the proof of Theorem~\ref{thm5.2.1}.

\subsection{Sanity check: comparison to previously known results}\label{s5.5}

As already mentioned in Subsection~\ref{s3.6}, it is known that 
all elements of the form $w_0^\mathfrak{p}w_0$, where $\mathfrak{p}$ is 
a parabolic subalgebra  of $\mathfrak{sl}_n$, are Kostant positive.

If we take $\mathfrak{p}$ to be a maximal parabolic subalgebra, that is,
one for which the semisimple part of the Levi factor equals
$\mathfrak{sl}_i\oplus \mathfrak{sl}_{n-i}$, for $i=0,1,\dots,\lfloor\frac{n}{2}\rfloor$,
the element $w_0^\mathfrak{p}w_0$ turns out to be fully commutative.
The Temperley-Lieb diagram of the element $w_0^\mathfrak{p}w_0$ is
as follows, where $i$ is the number of caps:
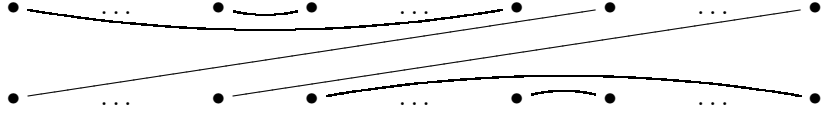
\begin{displaymath}
\xymatrix{ 
\bullet\ar@/_3mm/@{-}[rrrrr]&\dots&\bullet\ar@/_1mm/@{-}[r]&
\bullet&\dots&\bullet&\bullet\ar@{-}[dllllll]&\dots&\bullet\ar@{-}[dllllll]\\
\bullet&\dots&\bullet&\bullet\ar@/^3mm/@{-}[rrrrr]&\dots
&\bullet\ar@/^1mm/@{-}[r]&\bullet&\dots&\bullet
} 
\end{displaymath}
If $i=\frac{n}{2}$, this element is an involution, in fact, it is $\sigma_{a,b}$,
where $a=\frac{n}{2}$ and $b=\frac{n}{2}-1$. If $i<\frac{n}{2}$, 
the above element belongs to the left Kazhdan--Lusztig cell of $\sigma_{a,b}$, 
where $a=i$ and $b=i-1$. Therefore, for such elements, our Theorem~\ref{thm5.2.1}
agrees with the previous results.

Moreover, all known results in small ranks mentioned in 
Subsection~\ref{s3.6} indeed agree with Theorem~\ref{thm5.2.1}.

We also note that Theorem~\ref{thm5.2.1}, combined with \cite[Theorem~1.1]{Ka},
gives a lot of new full answers to Kostant's problem even for not necessarily
fully commutative permutations.

\subsection{Problems to extend outside fully commutative elements}\label{s5.6}

The fact that the answers to
Kostant's problem for the elements $s_1s_2s_1$ and $s_2s_3s_2$ of $S_5$
are different, see \cite[Section~4]{KaM}, suggests that it will not be straightforward to 
extend Theorem~\ref{thm5.2.1} outside the set of fully commutative elements.

\section{Asymptotic results}\label{s7}

\subsection{Various sequences}\label{s7.1}

For $n\in \mathbb{Z}_{\geq 1}$,  we denote
\begin{itemize}
\item by $\mathbf{ki}_n$ the number of fully commutative 
Kostant involutions in $S_n$;
\item by $\mathbf{k}_n$ the number of fully commutative 
$w\in S_n$ for which $\mathbf{K}(w)=\mathtt{true}$;
\item by $\mathbf{mi}_n$ the number of fully commutative 
involutions in $S_n$;
\item by $\mathbf{m}_n$ the number of fully commutative 
elements in  $S_n$.
\end{itemize}
For $a\in \mathbb{Z}_{\geq 1}$ and $\mathbf{x}\in\{\mathbf{ki}_n,
\mathbf{k}_n,\mathbf{mi}_n,\mathbf{m}_n\}$, we denote by 
$\mathbf{x}^a$ the number of elements in the family $\mathbf{x}$
with exactly $a$ caps.

It is very well-known that $\mathbf{m}_n$ equals 
the $n$-th Catalan number 
\begin{displaymath}
C_n:=\frac{(2n)!}{n!(n+1)!}=\frac{1}{n+1}\binom{2n}{n}.
\end{displaymath}

\subsection{Main asymptotic results}\label{s7.2}

\begin{theorem}\label{thm5.6.1}
{\hspace{1mm}}

\begin{enumerate}[$($a$)$]
\item\label{thm5.6.1-1} We have 
$\displaystyle\lim_{n\to\infty}\frac{\mathbf{ki}_n}{\mathbf{mi}_n}= 0$. 
\item\label{thm5.6.1-2} We have 
$\displaystyle\lim_{n\to\infty}\frac{\mathbf{k}_n}{\mathbf{m}_n}= 0$. 
\item\label{thm5.6.1-3} For any fixed $a\in\{0,1,\dots,\lfloor\frac{n}{2}\rfloor\}$, we have 
$\displaystyle\lim_{n\to\infty}\frac{\mathbf{ki}^a_n}{\mathbf{mi}^a_n}= 1$. 
\end{enumerate}
\end{theorem}

The remainder of this section is devoted to the proof of this
theorem. In Subsections~\ref{s7.3}, \ref{s7.4} and \ref{s7.5}, 
we first establish some explicit formulae
for the enumeration of the main protagonists defined in the previous
subsection. We are sure that some of the combinatorial arguments 
and results presented in this section are not new and can be 
found in or derived from the existing literature. However, we feel
that is would be more difficult to find appropriate references
than to prove these results. When working on the proofs,
the Online Encyclopedia of Integer Sequences was really helpful.

\subsection{Fully commutative Kostant involutions}\label{s7.3}

Recall the family of Fibonacci polynomials $F_n(x)$, where
$n\geq 0$, given by the following recursion:
\begin{displaymath}
F_0(x)=1,\qquad
F_1(x)=x,\qquad
F_n(x)=xF_{n-1}(x)+F_{n-2}(x),\quad n\geq 2.
\end{displaymath}
Here are some initial members of this family:
\begin{displaymath}
\begin{array}{c||c|c|c|c|c|c}
n:&0&1&2&3&4&5\\
\hline\hline
F_n(x):&1&x&x^2+1&x^3+2x&x^4+3x^2+1&x^5+4x^3+3x
\end{array}
\end{displaymath}
The evaluation $F_n(1)$ is exactly the $n$-th Fibonacci number.

\begin{proposition}\label{prop7.3-1}
We have $\displaystyle
F_n(x)=\sum_{a=0}^{\lfloor\frac{n}{2}\rfloor}\mathbf{ki}_n^a\cdot x^{n-2a}$.
\end{proposition}

\begin{proof}
This is easy to check for $n=1,2$. Therefore, it suffices to show that 
the numbers $\mathbf{ki}_n^a$ satisfy the same recursion as the coefficients
of the Fibonacci polynomials. 
{Observe that $\mathbf{ki}_n^a=0$ for $a> \lfloor\frac{n}{2}\rfloor$.}

Let $d$ be a Kostant involution and look at the strand in $\mathtt{e}_d$
which starts at the top point $1$. If this strand is vertical,
removing it yields a Kostant involution for $n-1$ with the same number
of caps as $\mathtt{e}_d$. If this strand is a cup, removing it together with the corresponding
cap produces a Kostant involution for $n-2$ with one fewer cap than $\mathtt{e}_d$.
This defines a bijection between the set of all 
Kostant involutions for $n$ with $a$ caps and the union of the
set of  all Kostant involutions for $n-1$ with $a$ caps and the set of all 
Kostant involutions for $n-2$ with $a-1$ caps.
This implies $\mathbf{ki}_{n}^a=\mathbf{ki}_{n-1}^a+\mathbf{ki}_{n-2}^{a-1}$,
which establishes the necessary recursion.
\end{proof}

\begin{corollary}\label{cor7.3-2}
The number $\mathbf{ki}_n$ is the $n$-th Fibonacci number. 
\end{corollary}

\begin{proof}
Evaluate the equality in Proposition~\ref{prop7.3-1} at $1$. 
\end{proof}

\begin{corollary}\label{cor7.3-3}
We have $\mathbf{ki}_n^a=\binom{n-a}{a}$. 
\end{corollary}

\begin{proof}
Again, this is easy to verify for small values of $n$, so we only 
need to check that the binomial coefficients on the right hand
side satisfy the same recursion as $\mathbf{ki}_n^a$. By the usual Pascal triangle formula,
we have:
\begin{displaymath}
\binom{n-a}{a}=\binom{(n-1)-a}{a}+\binom{n-1-a}{a-1}=
\binom{(n-1)-a}{a}+\binom{(n-2)-(a-1)}{a-1}.
\end{displaymath}
This implies the claim.
\end{proof}

\subsection{Fully commutative Kostant elements}\label{s7.4}

{
Under the Robinson-Schensted correspondence, 
two-sided Kazhdan--Lusztig cells in type $A$ are in bijection with partitions 
of $n$, and the left cells in a given two-sided cell corresponding to $\lambda\vdash n$ 
are in bijection with the standard tableaux of shape $\lambda$. More specifically, 
the Robinson-Schensted correspondence gives a bijection between the elements $w$ in a 
left cell and pairs $(P(w),Q(w))$ of standard Young tableaux of shape $\lambda$ for which $Q(w)$ is fixed, by~\cite[Theorem~1.4]{KL} (for 
the present formulation of that result and a more elementary proof, 
see~\cite[Theorem~A]{A00}). In particular, two-sided cells of fully commutative permutations correspond to partitions with at most two rows and the value of the $\mathbf{a}$-function is given by the length of the second row. This follows from, for example,  \cite[Lemma~6.5]{MT}.}

\begin{corollary}\label{cor7.4-3}
We have $\mathbf{k}_n^a=\binom{n-a}{a}\frac{n!(n-2a+1)!}{a!(n-2a)!(n-a+1)!}$. 
\end{corollary}

\begin{proof}
%In type $A$, two-sided Kazhdan--Lusztig cells are in bijection with 
%partitions of $n$, see \cite{KL}. 
As recalled above, the two-sided cell which corresponds to our $a$ is indexed 
by the partition $(n-a,a)$. Since every Kazhdan--Lusztig left cell contains a unique involution,
the number $\mathbf{k}_n^a$ is the product of $\mathbf{ki}_n^a$
with the size of this left cell. The latter equals the number of standard
Young tableaux {$P(w)$} of shape $(n-a,a)$ {under the Robinson-Schensted correspondence.} 
%{ because the Robinson-Schensted correspondence gives a 
%bijection between the elements $w$ in the left cell and pairs $(P(w),Q(w))$ 
%of such Young tableaux for which $Q(w)$ is fixed, by~\cite[Theorem~1.4]{KL} (for 
%the present formulation of that result and a more elementary proof, 
%see~\cite[Theorem~A]{A00}).} 
Now the claim of our corollary follows from Corollary~\ref{cor7.3-3} and the Hook Formula.
\end{proof}

\subsection{Fully commutative involutions}\label{s7.5}

\begin{proposition}\label{cor7.5-1}
We have $\mathbf{mi}_n=\binom{n}{\lfloor\frac{n}{2}\rfloor}$. 
\end{proposition}

\begin{proof}
This claim is easy to check for small values of $n$, so we need to 
show that both sides satisfy the same recursion. 

Let $d$ be a fully commutative involution. Consider the strand of $\mathtt{e}_d$
connected to the upper point $1$. If it is vertical, removing it results
in a fully commutative involution for $n-1$. If the strand is a
cup $X$ connecting $1$ to some point $2i$, there is 
a crossingless pairing of $2i-2$ points inside this cup. 
Removing $X$ and all cups contained inside it together with
the corresponding caps, we obtain a fully commutative involution
for $n-2i$. This implies that we have the following recursion for the
left hand side of our formula:
\begin{displaymath}
\mathbf{mi}_n=\mathbf{mi}_{n-1}+
\sum_{i=1}^{\lfloor\frac{n}{2}\rfloor}C_{i-1}\mathbf{mi}_{n-2i}.
\end{displaymath}

We claim the the middle binomial coefficients 
$\binom{n}{\lfloor\frac{n}{2}\rfloor}$ satisfy the same recursion.
Indeed, consider the Pascal triangle with the point $(x,y)$ 
corresponding to $\binom{x}{y}$, for all appropriate $x$ and $y$. 
Then $\binom{n}{i}$ is exactly the number of shortest paths between 
$(0,0)$ and $(n,i)$ in this triangle.

Assume first that $n=2k$. Then we need to prove that 
\begin{displaymath}
\binom{2k}{k}=\binom{2k-1}{k-1}+
\sum_{i=1}^{k}C_{i-1}\binom{2k-2i}{k-i}.
\end{displaymath}
Equivalently, we need to show that
\begin{displaymath}
\binom{2k-1}{k}= \sum_{i=1}^{k}C_{i-1}\binom{2k-2i}{k-i}.
\end{displaymath}
For $i=1,2,\dots,k$, let $\mathcal{P}_i$ be the set of all
shortest paths between $(0,0)$ and $(2k-1,k)$ such that 
the path goes through the point $(2(k-i),k-i)$ and $i$ is 
minimal with this property. The set of all paths between
$(0,0)$ and $(2k-1,k)$ is the disjoint union of the 
$\mathcal{P}_i$. There are $\binom{2k-2i}{k-i}$
shortest paths between $(2(k-i),k-i)$ and $(0,0)$.
The classical interpretation of Catalan numbers as paths
in a square
that do not cross the diagonal implies that  there
are exactly $C_{i-1}$ shortest paths between 
$(2k-1,k)$ and $(2(k-i),k-i)$ satisfying the
condition for the minimality of $i$. This establishes
the necessary recursion formula for the right hand side.

The case $n=2k+1$ is similar and left to the reader.
\end{proof}

Darij Grinberg informed us that the formula in Proposition~\ref{cor7.5-1}
can be found in \cite[Proposition~3]{SS} with a different proof.

\subsection{Proof of Theorem~\ref{thm5.6.1}\eqref{thm5.6.1-3}}\label{s7.6}

%{
%Recall that, under the Robinson-Schensted correspondence, 
%fully commutative permutations correspond to partitions with
%at most two rows. For such partition, the value of the 
%$\mathbf{a}$-function is given by the length of the 
%second row. This follows from, for example,  \cite[Lemma~6.5]{MT}.}
{ As remarked at the beginning of Section~\ref{s7.4}, the partition corresponding to $a$ is 
$(n-a,a)$, hence} $\mathbf{mi}_n^a=\frac{n!(n-2a+1)!}{a!(n-2a)!(n-a+1)!}$
by the Hook Formula. Using Corollary~\ref{cor7.3-3}, we have
\begin{displaymath}
\frac{\mathbf{ki}_n^a}{\mathbf{mi}_n^a}=
\frac{(n-a)!a!(n-2a)!(n-a+1)!}{a!(n-2a)!n!(n-2a+1)!}
=\frac{(n-a+1)(n-a)\dots(n-2a+2)}{n(n-1)\dots(n-a+1)}.
\end{displaymath}
Here both the numerator and the denominator are polynomials in
$n$ of degree $a$ and with leading coefficient $1$. The claim 
of Theorem~\ref{thm5.6.1}\eqref{thm5.6.1-3} follows.

\subsection{Proof of Theorem~\ref{thm5.6.1}\eqref{thm5.6.1-1}}\label{s7.7}

By Corollary~\ref{cor7.3-2}, the number $\mathbf{ki}_n$ is the
$n$-th Fibonacci number. It is given by the formula
\begin{displaymath}
\frac{(\frac{1+\sqrt{5}}{2})^n-(\frac{1-\sqrt{5}}{2})^n}{\sqrt{5}}.
\end{displaymath}
Since the absolute value of  $\frac{1-\sqrt{5}}{2}$ is less than $1$,
it follows that $F_n$ grows as $(\frac{1+\sqrt{5}}{2})^n$.

At the same time, if $n=2k$, the { central binomial coefficient 
$\binom{2k}{k}$ is not smaller than $\frac{4^{k}}{2k+1}$,
as follows directly from
\begin{displaymath}
4^k=2^{2k}=(1+1)^{2k}=\sum_{i=0}^{2k}\binom{2k}{i}.
\end{displaymath}
If $n=2k+1$, the coefficient $\binom{2k+1}{k}$ is not smaller 
than $\binom{2k}{k}$. This implies that
$\binom{n}{\lfloor\frac{n}{2}\rfloor}$ grows at least as fast as} 
$\frac{2^n}{n}$. Since $\frac{1+\sqrt{5}}{2}<2$, the
claim of Theorem~\ref{thm5.6.1}\eqref{thm5.6.1-1} follows.

\subsection{Proof of Theorem~\ref{thm5.6.1}\eqref{thm5.6.1-2}}\label{s7.8}

Using Corollary~\ref{cor7.4-3}, we need to show that 
\begin{equation}\label{eq7.8-1}
\sum_{a=0}^{\lfloor\frac{n}{2}\rfloor} 
\frac{(n+1)\binom{n-a}{a}}{\binom{2n}{n}}\cdot\frac{n!(n-2a+2)!}{a!(n-2a)!(n-a+1)!}\to
0,\quad n\to\infty.
\end{equation}
We rewrite the expression in \eqref{eq7.8-1} as
\begin{displaymath}
\sum_{a=0}^{\lfloor\frac{n}{2}\rfloor} 
\frac{(n+1)(n-a)!n!(n-2a+1)!n!n!}{a!(n-2a)!a!(n-2a)!(n-a+1)!(2n)!}
\end{displaymath}
and then, further, as
\begin{displaymath}
\sum_{a=0}^{\lfloor\frac{n}{2}\rfloor} 
\frac{(n+1)(n-a)!(n-a)!n!(n-2a+1)!n!n!a!}{a!(n-2a)!a!(n-2a)!(n-a+1)!(2n)!a!(n-a)!}
\end{displaymath}
and, finally, as 
\begin{equation}\label{eq7.8-2}
\sum_{a=0}^{\lfloor\frac{n}{2}\rfloor} 
\frac{(n+1)\binom{n-a}{a}\binom{n-a}{a}\binom{n}{a}}{\binom{n-a+1}{a}\binom{2n}{n}}.
\end{equation}
Note that $\binom{n-a}{a}\leq \binom{n-a+1}{a}$ and hence the expression in
\eqref{eq7.8-2} is bounded from above by 
\begin{equation}\label{eq7.8-3}
\sum_{a=0}^{\lfloor\frac{n}{2}\rfloor} 
\frac{(n+1)\binom{n-a}{a}\binom{n}{a}}{\binom{2n}{n}}.
\end{equation}
The Fibonacci coefficient $\binom{n-a}{a}$ is bounded by the
$n$-the Fibonacci number and hence grows at most as $(\frac{1+\sqrt{5}}{2})^n$.
The coefficient $\binom{n}{a}$ is bounded by 
$\binom{n}{\lfloor\frac{n}{2}\rfloor}$ and hence grows as $2^n$
up to some factor of at most polynomial growth. At the same time,
$\binom{2n}{n}$ grows as $4^n$ up to some factor of at most polynomial growth. 
As the number of summands is linear in $n$, it follows that the
whole expression \eqref{eq7.8-3} tends to $0$, when $n\to\infty$.
This proves Theorem~\ref{thm5.6.1}\eqref{thm5.6.1-2}.

\subsection{Conjectures}\label{s7.9}

Taking Theorem~\ref{thm5.6.1} into account, we conjecture the following
for general elements of $S_n$:
\begin{itemize}
\item Among all involutions in $S_n$, the proportion of those 
for which the answer to Kostant's problem is positive is asymptotically $0$.
\item The proportion of elements in $S_n$ for which the answer to
Kostant's problem is positive is asymptotically $0$.
\item Fix a partition $\lambda$ of some $m$ and consider, for $n>m$, 
the partition $\lambda^{(n)}$ of $n$ obtained from $\lambda$ by
increasing the first part by $n-m$. Then, among the elements
in $S_n$ belonging to the two-sided cell indexed
by the partition $\lambda^{(n)}$, the proportion of those for which the answer 
to Kostant's problem is positive  is asymptotically $1$.
\end{itemize}
{
We note that the recent results in \cite[Subsection~4.10]{MSr} support, 
in some mild sense, these conjectures.
}

\section{Kostant's problem and Barbasch--Vogan theorem for fiab bicategories}\label{s6}

\subsection{Fiat $2$-categories}\label{s6.1}

Let $\cC$ be a fiab bicategory in the sense of \cite{MMMTZ},
{ i.e. a finitary bicategory with weak
involution $\star$ and adjunction morphisms}. Consider the 
left, right and two-sided pre-orders $\leq_L$, $\leq_R$ and $\leq_J$
on the set $\mathcal{S}(\cC)$ of isomorphism classes of indecomposable
$1$-morphisms in $\cC$. { In particular, we have 
$\mathrm{F}\leq_L\mathrm{G}$ provided that there exists a $1$-morphism 
$\mathrm{H}$ such that $\mathrm{G}$ is isomorphic to a direct summand of 
$\mathrm{H}\mathrm{F}$. The other pre-orders are defined similarly,} 
see \cite[Section~3]{MM2} for details.

The associated equivalence classes are called cells (left, right and
two-sided, respectively). Each left and each right cell contains 
a unique special $1$-morphism called the Duflo $1$-morphism,
see \cite[Subsection~4.5]{MM1}.

\subsection{Kostant's problem for fiat bicategories}\label{s6.7}

Fix a left cell $\mathcal{L}$ in $\cC$. Then there is an
object $\mathtt{i}\in\cC$ such that all elements of $\mathcal{L}$
have $\mathtt{i}$ as a domain. Consider the abelianization
$\overline{\mathbf{P}}_\mathtt{i}$ of the principal
birepresentation ${\mathbf{P}}_\mathtt{i}$ of $\cC$
in the sense of \cite[Section~3]{MMMT}.

Denote by $\hat{\mathcal{L}}$ the set of all 
$\mathrm{F}\in \mathcal{S}(\cC)$ such that $\mathrm{F}\leq_L\mathcal{L}$. 
The collection of Serre subcategories of the 
${\mathbf{P}}_\mathtt{i}(\mathtt{j})$, for $\mathtt{j}\in \cC$,
generated by all simple objects $L_{\mathrm{F}}$, where $\mathrm{F}\in \hat{\mathcal{L}}$,
is invariant under the action of $\cC$. We denote the corresponding (abelian)
birepresentation of $\cC$ by $\overline{\mathbf{M}}_{\mathcal{L}}$
and its (finitary) restriction to projective objects by ${\mathbf{M}}_{\mathcal{L}}$.

On the other hand, let $\mathrm{D}_\mathcal{L}$ be the Duflo $1$-morphism
in $\mathcal{L}$ and consider the finitary birepresentation
${\mathbf{K}}_{\mathcal{L}}$ given by the action of $\cC$ on
the additive closure of $\cC L_{\mathrm{D}_\mathcal{L}}$.
The Yoneda morphism from $\mathbf{P}_{\mathtt{i}}$ to
${\mathbf{K}}_{\mathcal{L}}$ sending $\mathbbm{1}_{\mathtt{i}}$
to $L_{\mathrm{D}_\mathcal{L}}$ factors through ${\mathbf{M}}_{\mathcal{L}}$
by construction. We will say that $\mathrm{D}_\mathcal{L}$ is
{\em Kostant positive} provided that the induced morphism of birepresentations
from ${\mathbf{M}}_{\mathcal{L}}$ to ${\mathbf{K}}_{\mathcal{L}}$ is an equivalence.

{
Let us now look what happens in the case of the bicategory $\cP$.  
The identity object of $\cP$ is given by the quotient of the 
universal enveloping algebra modulo the annihilator of $\mathcal{O}_0$.
By \cite[Theorem~5.9]{BG}, $\mathcal{O}_0$ itself is equivalent to the 
subcategory of the principal birepresentation of 
$\cP$ generated by the quotient of the universal enveloping algebra
by the central character of $\mathcal{O}_0$.
Hence, we can view a simple module $L_w\in \mathcal{O}_0$
as a simple object of the principal birepresentation of $\cP$.
Take now $w=d$ to be the Duflo involution in some 
Kazhdan-Lusztig left cell $\mathcal{L}$.
Then the corresponding birepresentation ${\mathbf{M}}_{\mathcal{L}}$ 
as defined above is generated, as a birepresentation of $\cP$, by the
quotient of the universal enveloping algebra by the 
annihilator of $L_d$. At the same time, in
\cite[Proposition~7.2]{KMM} it is shown that the 
corresponding birepresentation 
${\mathbf{K}}_{\mathcal{L}}$ as defined above is generated by
$\mathcal{L}(L_p,L_p)$. Therefore, the fact that that
the natural embedding of ${\mathbf{M}}_{\mathcal{L}}$ into
${\mathbf{K}}_{\mathcal{L}}$ is an equivalence
is, indeed, equivalent to the fact that the answer to Kostant's 
problem for $L_d$ is positive, see \cite[Corollary~7.6]{KMM}.
}

Note that the kernel of the above Yoneda morphism from $\mathbf{P}_{\mathtt{i}}$ to
${\mathbf{K}}_{\mathcal{L}}$ is exactly the annihilator of $L_{\mathrm{D}_\mathcal{L}}$
in $\cC$, i.e. the left biideal of $2$-morphisms $\alpha$ in $\cC$ such that
$\mathbf{P}_{\mathtt{i}}(\alpha)_{L_{\mathrm{D}_\mathcal{L}}}=0$. 
This connects our reformulation of Kostant's problem in this more general
setup to the problem of studying annihilators of simple objects in 
principal birepresentations. This naturally leads to 
an analogue of the classical Barbasch--Vogan theorem for fiab bicategories,
presented in the next subsection.

\subsection{Barbasch--Vogan theorem for $\cC$}\label{s6.4}

Note that the abelianization $\overline{\mathbf{P}}_\mathtt{i}$ of 
${\mathbf{P}}_\mathtt{i}$ is an abelian birepresentation of $\cC$.
For $\mathtt{j}\in\cC$ and an indecomposable $1$-morphism
$\mathrm{F}\in\cC(\mathtt{i},\mathtt{j})$, 
denote by $\cJ_{\mathrm{F}}$ the annihilator
of $L_{\mathrm{F}}$ in $\cC$.
Then $\cJ_{\mathrm{F}}$ is a left biideal of $\cC$.

\begin{theorem}\label{thm6.4-1}
For two  indecomposable $1$-morphism $\mathrm{F}$ and $\mathrm{G}$
in $\cC$, the following conditions are equivalent:
\begin{enumerate}[$($a$)$]
\item \label{thm6.4-1-1} $\mathrm{F}\leq_R \mathrm{G}$.
\item \label{thm6.4-1-2} $\cJ_{\mathrm{G}}\subseteq \cJ_{\mathrm{F}}$ 
% \item \label{thm6.4-1-3} 
% $\overline{\cJ}_{\mathrm{G}}\subseteq \overline{\cJ}_{\mathrm{F}}$.
\end{enumerate}
\end{theorem}

\begin{proof}
% We will prove the equivalence of \eqref{thm6.4-1-1}
% and \eqref{thm6.4-1-2}. 
% The equivalence of \eqref{thm6.4-1-1}
% and \eqref{thm6.4-1-3} is proved in a similar way.

In this proof, we will use the word ``module'' instead of 
``birepresentation''. The direct sum of all principal $\cC$-modules
has the obvious structure of a $\cC$-$\cC$-bimodule, the regular
$\cC$-$\cC$-bimodule. 

Given an object $M$ of this regular $\cC$-$\cC$-bimodule and any 
$1$-morphism $\mathrm{F}\in\cC$ which acts on $M$ on the right, we have
\begin{displaymath}
\mathrm{Ann}_{\ccC}(M)\subset  \mathrm{Ann}_{\ccC}(M\mathrm{F})
\end{displaymath}
of left annihilators.
Also, if some $L_{\mathrm{G}}$ is a subquotient of $M$, we have
\begin{displaymath}
\mathrm{Ann}_{\ccC}(M)\subset  \mathrm{Ann}_{\ccC}(L_{\mathrm{G}}).
\end{displaymath}

Let $\mathrm{F}$ and $\mathrm{G}$ be two $1$-morphisms in $\cC$. Then 
$\mathrm{F}\leq_R \mathrm{G}$ if and only if there exists
$\mathrm{H}$ in $\cC$ such that $\mathrm{G}$ is a summand of $\mathrm{F}\mathrm{H}$.
By adjunction, this is equivalent to
\begin{displaymath}
0\neq\mathrm{Hom}_{\overline{\ccC}}(\mathrm{F}\mathrm{H},L_{\mathrm{G}})
\cong \mathrm{Hom}_{\overline{\ccC}}(\mathrm{F},L_{\mathrm{G}}\mathrm{H}^\star).
\end{displaymath}
In other words, $\mathrm{F}\leq_R \mathrm{G}$ is equivalent to the existence
of $\mathrm{H}$ such that $L_{\mathrm{F}}$ is a subquotient of 
$L_{\mathrm{G}}\mathrm{H}^\star$. The previous paragraph now yields
that $\mathrm{F}\leq_R \mathrm{G}$ implies 
$\cJ_{\mathrm{G}}\subset \cJ_{\mathrm{F}}$.

For the converse, note that if $\mathrm{F}\nleq_R \mathrm{G}$, then either $\mathrm{F}>_R \mathrm{G}$, or $\mathrm{F}$ and $\mathrm{G}$ are not comparable in the right order.

If $\mathrm{G}<_R \mathrm{F}$, then all elements in the
two-sided cell of $\mathrm{F}$ annihilate $L_{\mathrm{G}}$ by
\cite[Lemma~12]{MM1}. By the same lemma, there are elements in the 
two-sided cell of $\mathrm{F}$ that do not annihilate $L_{\mathrm{F}}$.
Therefore $\cJ_{\mathrm{F}}\subsetneq \cJ_{\mathrm{G}}$.

If $\mathrm{F}$ and $\mathrm{G}$ are not comparable in the right order, then the
two-sided cells of $\mathrm{F}$  and $\mathrm{G}$ annihilate $L_\mathrm{G}$ and $L_\mathrm{F}$, 
respectively, again by \cite[Lemma~12]{MM1}. On the other hand, the two sided cell of $\mathrm{F}$ 
does not annihilate $L_\mathrm{F}$, and similarly that of $\mathrm{G}$ does not annihilate 
$L_\mathrm{G}$. Thus the annihilators are not comparable.

This completes the proof.
\end{proof}

We can also give slightly more detailed information.

\begin{proposition}\label{prop6.4-2}
In the setup of  Theorem~\ref{thm6.4-1}, if 
$\mathrm{F}<_R \mathrm{G}$ and $\mathtt{i}$ is the codomain
for both $\mathrm{F}$ and $\mathrm{G}$, then 
there is $\mathrm{H}\in\cC(\mathtt{i},\mathtt{i})$ such that 
\begin{displaymath}
\dim \mathrm{Hom}_{\ccC/\ccJ_{\mathrm{G}}}(\mathrm{H},\mathbbm{1}_{\mathtt{i}}) 
\neq
\dim\mathrm{Hom}_{\ccC/\ccJ_{\mathrm{F}}}(\mathrm{H},\mathbbm{1}_{\mathtt{i}}) .
\end{displaymath} 
\end{proposition}

\begin{proof}
We take $\mathrm{H}$ to be the Duflo $1$-morphism in the 
right cell of $\mathrm{G}$. Then $\mathrm{H}L_{\mathrm{F}}=0$
by \cite[Lemma~12]{MM1} and hence the evaluation of any
element of $\mathrm{Hom}_{\ccC}(\mathrm{H},\mathbbm{1}_{\mathtt{i}})$
at $L_{\mathrm{F}}$ is the zero morphism. At the same time, the evaluation of
the morphism $\mathrm{H}\to \mathbbm{1}_{\mathtt{i}}$ which defines
$\mathrm{H}$ as a Duflo $1$-morphism
(see \cite[Subsection~4.5]{MM1}) at $L_{\mathrm{G}}$ is non-zero.
The claim follows.
\end{proof}

\subsection{Classical Barbasch--Vogan theorem}\label{s6.5}

The above proposition implies the following classical
result due to Barbasch and Vogan, see \cite{BV1,BV2}.

\begin{corollary}\label{cor6.5-1}
Let $\mathfrak{g}$ be a semi-simple finite dimensional complex 
Lie algebra with Weyl group $W$. Then, for $x,y\in W$, we have
$\mathrm{Ann}_{U(\mathfrak{g})}(L_x)\subseteq \mathrm{Ann}_{U(\mathfrak{g})}(L_y)$
if and only if $y\leq_L x$, where $\leq_L$ is the Kazhdan--Lusztig
left order on $W$.
\end{corollary}

The appearance of the left order in Corollary~\ref{cor6.5-1},
compared to the right order in Theorem~\ref{thm6.4-1}, is due to 
the right nature of the action of the bicategory of
projective functors on  category $\mathcal{O}$.

\begin{proof}
Consider the  bicategory $\cP$ of projective functors acting on 
$\mathcal{O}_0$. The latter is, naturally, a subbirepresentation
of the abelianized principal birepresentation $\overline{\mathbf{P}}$
and is equivalent to a certain category of Harish-Chandra
bimodules for $\mathfrak{g}$ by \cite[Theorem~5.9]{BG}. 
This equivalence matches the indecomposable projective object $P_{\theta_e}$
corresponding to the identity with the quotient of 
$U(\mathfrak{g})$ modulo the trivial central character. 

Sending $P_{\theta_e}$ to the dominant Verma module
in $\mathcal{O}_0$ using the Yoneda Lemma, defines a morphism 
of birepresentations from $\overline{\mathbf{P}}$
to $\mathcal{O}_0$ which sends simple objects to simple objects.
By  Theorem~\ref{thm6.4-1}, $y\leq_L x$ implies
$\mathrm{Ann}_{\ccP}(L_x)\subset \mathrm{Ann}_{\ccP}(L_y)$.

In particular, 
\begin{displaymath}
\mathrm{Hom}_{\mathrm{Ann}_{\ccP}(L_x)} 
\big(\bigoplus_{w\in W}\theta_w,\theta_e\big)\subseteq
\mathrm{Hom}_{\mathrm{Ann}_{\ccP}(L_y)} 
\big(\bigoplus_{w\in W}\theta_w,\theta_e\big)\subseteq
\mathrm{Hom}_{\ccP} 
\big(\bigoplus_{w\in W}\theta_w,\theta_e\big),
\end{displaymath}
where the latter corresponds to the quotient of $U(\mathfrak{g})$
modulo the ideal generated by the 
trivial central character under the equivalence from \cite[Theorem~5.9]{BG}. 
Hence $\mathrm{Ann}_{U(\mathfrak{g})}(L_x)\subseteq \mathrm{Ann}_{U(\mathfrak{g})}(L_y)$.

Since $\cP$ has only one object, the fact that $y<_L x$ implies 
$\mathrm{Ann}_{U(\mathfrak{g})}(L_x)\subsetneq \mathrm{Ann}_{U(\mathfrak{g})}(L_y)$
follows directly from Proposition~\ref{prop6.4-2}. As in the proof of Theorem \ref{thm6.4-1}, 
if $x$ and $y$ are not comparable in the left order, 
their annihilators are incomparable, which completes the proof.
\end{proof}

%\vspace{2mm}

\noindent
M.~M.: Center for Mathematical Analysis, Geometry, and Dynamical Systems, Departamento de Matem{\'a}tica, 
Instituto Superior T{\'e}cnico, 1049-001 Lisboa, PORTUGAL \& Departamento de Matem{\'a}tica, FCT, 
Universidade do Algarve, Campus de Gambelas, 8005-139 Faro, PORTUGAL, email: {\tt mmackaay\symbol{64}ualg.pt}

\noindent
Vo.~Ma.: Department of Mathematics, Uppsala University, Box. 480,
SE-75106, Uppsala, SWEDEN, email: {\tt mazor\symbol{64}math.uu.se}

\noindent
Va.~Mi.: School of Mathematics, University of East Anglia, Norwich, 
NR4 7TJ, UK, email: {\tt V.Miemietz\symbol{64}uea.ac.uk}


\begin{thebibliography}{9999999}
{
\bibitem[A00]{A00} 
Ariki, S. Robinson-Schensted correspondence and left cells. Advanced Studies in Pure Mathematics 28, 2000, Combinatorial Methods in Representation Theory 
pp. 1--20.}
\bibitem[BV82]{BV1} Barbasch, D.; Vogan, D. Primitive ideals and orbital 
integrals in complex classical groups. Math. Ann. {\bf 259} (1982), no. 2, 153--199.
\bibitem[BV83]{BV2} Barbasch, D.; Vogan, D. Primitive ideals and orbital 
integrals in complex exceptional groups. J. Algebra {\bf 80} (1983), no. 2, 350--382. 
\bibitem[BB81]{BB81} Beilinson, A.; Bernstein, J. Localisation de $\mathfrak{g}$-modules. 
C. R. Acad. Sci. Paris Ser. I Math. {\bf 292} (1981), no. 1, 15--18.
\bibitem[BGG76]{BGG} Bernstein, I.; Gelfand, I.; Gelfand, S. 
A certain category of $\mathfrak{g}$-modules. (Russian) Funkcional. Anal. i Prilozen. 
{\bf 10} (1976), no. 2, 1--8.
\bibitem[{BG80}]{BG} Bernstein, J.; Gelfand, S. Tensor products of finite- and 
infinite-dimensional representations of semisimple Lie algebras. Compositio Math. 
{\bf 41} (1980), no. 2, 245--285.
\bibitem[BJS93]{BJS} Billey, S.; Jockusch, W.; Stanley, R. Some combinatorial properties 
of Schubert polynomials. J. Algebraic Combin. {\bf 2} (1993), no. 4, 345--374.
\bibitem[BS10]{BS} Brundan, J.; Stroppel, C. Highest weight categories arising from 
Khovanov's diagram algebra. II. Koszulity. Transform. Groups {\bf 15} (2010), no. 1, 1--45.
\bibitem[BS11]{BS2} Brundan, J.; Stroppel, C. Highest weight categories arising 
from Khovanov's diagram algebra III: category $\mathcal{O}$. Represent. Theory 
{\bf 15} (2011), 170--243.
\bibitem[BK81]{BK81} Brylinski, J.-L.; Kashiwara, M. Kazhdan--Lusztig conjecture 
and holonomic systems. Invent. Math. {\bf 64} (1981), no. 3, 387--410.
\bibitem[Du77]{Du} Duflo, M. Sur la classification des id{\'e}aux primitifs dans 
l'alg{\`e}bre enveloppante d'une alg{\`e}bre de Lie semi-simple. Ann. of Math. (2) 
{\bf 105} (1977), no. 1, 107--120.
\bibitem[GJ81]{GJ} Gabber, O.; Joseph, A. On the Bernstein-Gelfand-Gelfand resolution 
and the Duflo sum formula. Compositio Math. {\bf 43} (1981), no. 1, 107--131. 
\bibitem[Ge06]{Ge} Geck, M. Kazhdan--Lusztig cells and the Murphy basis. 
Proc. London Math. Soc. (3) {\bf 93} (2006), no. 3, 635--665.  
\bibitem[Hu08]{Hu} Humphreys, J. Representations of semisimple Lie algebras in the 
BGG category $\mathcal{O}$. Graduate Studies in Mathematics, {\bf 94}. 
American Mathematical Society, Providence, RI, 2008. xvi+289 pp.
\bibitem[Fa95]{Fa95} Fan, C. K. A Hecke Algebra Quotient and Properties of 
Commutative Elements of a Weyl Group, Ph.D. thesis, MIT, 1995.
\bibitem[Fa96]{Fa96} Fan, C. K. A Hecke algebra quotient and some combinatorial applications.
J. Algebraic Combin. {\bf 5} (1996), no. 3, 175--189.
{\bibitem[FG97]{FG97} Fan, C. K.; Green, R. M. Monomials and Temperley-Lieb algebras. J. Algebra {\bf 190} (1997), 498--517.}
\bibitem[Ja83]{Ja} Jantzen, J. Einh{\"u}llende Algebren halbeinfacher Lie-Algebren. (German) 
Ergebnisse der Mathematik und ihrer Grenzgebiete (3), {\bf 3}. Springer-Verlag, Berlin, 1983. ii+298 pp.
\bibitem[Jo80]{Jo} Joseph, A. Kostant's problem, Goldie rank and the Gelfand-Kirillov 
conjecture. Invent. Math. {\bf 56} (1980), no. 3, 191--213.
\bibitem[KL79]{KL} Kazhdan, D.; Lusztig, G. Representations of Coxeter groups and 
Hecke algebras. Invent. Math. {\bf 53} (1979), no. 2, 165--184. 
\bibitem[KhM04]{KhM} Khomenko, O.; Mazorchuk, V. Structure of modules induced from 
simple modules with minimal annihilator. Canad. J. Math. {\bf 56} (2004), no. 2, 293--309. 
\bibitem[KiM16]{KM} Kildetoft, T.; Mazorchuk, V. Parabolic projective functors in 
type $A$. Adv. Math. {\bf 301} (2016), 785--803.
\bibitem[Ka10]{Ka} K{\aa}hrstr{\"o}m, J. Kostant's problem and parabolic subgroups. 
Glasg. Math. J. {\bf 52} (2010), no. 1, 19--32.
\bibitem[KaM10]{KaM} K{\aa}hrstr{\"o}m, J.; Mazorchuk, V. A new approach to 
Kostant's problem. Algebra Number Theory {\bf 4} (2010), no. 3, 231--254.
\bibitem[KMM20]{KMM} Ko, H.; Mazorchuk, V.;  Mr{\dj}en, R.
Some homological properties of category $\mathcal{O}$, V.
Preprint arXiv:2007.00342, to appear in IMRN.
\bibitem[M-T19]{MMMT} Mackaay, M.; Mazorchuk, V.; Miemietz, V.; 
Tubbenhauer, D.; Simple transitive 2-representations via 
(co-)algebra 1-morphisms. Indiana Univ. Math. J. {\bf 68} (2019), no. 1, 1--33. 
\bibitem[M-Z21]{MMMTZ} Mackaay, M.; Mazorchuk, V.; Miemietz, V.; 
Tubbenhauer, D.; Zhang, X. Finitary birepresentations of finitary 
bicategories. Forum Math. {\bf 33} (2021), no. 5, 1261--1320. 
\bibitem[Mar91]{Mar} Martin, P. Potts models and related problems in statistical 
mechanics. Series on Advances in Statistical Mechanics, {\bf 5}. World Scientific 
Publishing Co., Inc., Teaneck, NJ, 1991. xiv+344 pp.
\bibitem[Mat64]{Mat} Matsumoto, H. G{\'e}n{\'e}rateurs et relations des groupes de 
Weyl g{\'e}n{\'e}ralis{\'e}s. C. R. Acad. Sci. Paris {\bf 258} (1964), 3419--3422. 
\bibitem[Ma05]{Ma} Mazorchuk, V. A twisted approach to Kostant's problem. 
Glasg. Math. J. {\bf 47} (2005), no. 3, 549--561. 
\bibitem[MM11]{MM1} Mazorchuk, V.; Miemietz, V. Cell 2-representations of 
finitary 2-categories. Compos. Math. {\bf 147} (2011), no. 5, 1519--1545.
\bibitem[MM14]{MM2} Mazorchuk, V.; Miemietz, V. Additive versus abelian 
2-representations of fiat 2-categories. Mosc. Math. J. {\bf 14} 
(2014), no. 3, 595--615.
{
\bibitem[MSr23]{MSr} Mazorchuk, V.; Srivastava, S. 
Kostant's problem for parabolic Verma modules.
Preprint arXiv:2301.07090.
}
\bibitem[MS08]{MS} Mazorchuk, V.; Stroppel, C. Categorification of (induced) cell modules and 
the rough structure of generalised Verma modules. Adv. Math. {\bf 219} (2008), no. 4, 1363--1426. 
%\bibitem[MS08b]{MS2} Mazorchuk, V.; Stroppel, C. Projective-injective modules, 
%Serre functors and symmetric algebras. J. Reine Angew. Math. {\bf 616} (2008), 131--165.
{
\bibitem[MT22]{MT} Mazorchuk, V.; Tenner, B.~E. Intersecting principal Bruhat 
ideals and grades of simple modules. Comb. Theory {\bf 2} (2022), no. 1, Paper No. 14, 31 pp.
}
\bibitem[MiSo97]{MiSo} Mili{\v c}i{\'c}, D.; Soergel, W. The composition series of 
modules induced from Whittaker modules. Comment. Math. Helv. {\bf 72} (1997), no. 4, 503--520. 
\bibitem[Sa01]{Sa} Sagan, B. The symmetric group. Representations, combinatorial algorithms, 
and symmetric functions. Second edition. Graduate Texts in Mathematics, {\bf 203}. 
Springer-Verlag, New York, 2001. xvi+238 pp. 
\bibitem[Sc61]{Sc} Schensted, C. Longest increasing and decreasing subsequences.
Canadian J. Math. {\bf 13} (1961), 179--191. 
\bibitem[SS84]{SS} Simion, R.; Schmidt, F. Restricted permutations. Proceedings of the 
fifteenth Southeastern conference on combinatorics, graph theory and computing 
(Baton Rouge, La., 1984). Congr. Numer. {\bf 45} (1984), 323--325. 
\bibitem[So90]{So} Soergel, W. Kategorie $\mathcal{O}$, perverse Garben und Moduln 
{\"u}ber den Koinvarianten zur Weylgruppe. J. Amer. Math. Soc. {\bf 3} (1990), no. 2, 421--445. 
\bibitem[So92]{So3} Soergel, W. The combinatorics of Harish-Chandra bimodules. 
J. Reine Angew. Math. {\bf 429} (1992), 49--74. 
\bibitem[So07]{So2} Soergel, W. Kazhdan--Lusztig-Polynome und unzerlegbare Bimoduln 
{\"u}ber Polynomringen. (German) J. Inst. Math. Jussieu {\bf 6} (2007), no. 3, 501--525.
\bibitem[St90]{St} Stroppel, C. Category $\mathcal{O}$: gradings and translation functors. 
J. Algebra {\bf 268} (2003), no. 1, 301--326. 
\end{thebibliography}
\end{document}